\documentclass{ws-m3as}

\usepackage{amsmath,amsfonts, amssymb}
\usepackage{pifont,mathtools,mathrsfs}
\usepackage{enumitem,enumerate}
\usepackage{algorithm,algpseudocode}

\usepackage[pdftex,dvipsnames]{xcolor}
\usepackage{psfrag, graphicx,float,indentfirst,subcaption}
\usepackage{mathbbol}
\usepackage{booktabs, cellspace, hhline}

\usepackage[super]{cite}
\usepackage{hyperref} %
\hypersetup{
	plainpages=false,
    colorlinks,
    linktocpage=true,   %
    linkcolor={red!50!black},
    citecolor={blue!50!black},
    urlcolor={blue!80!black}
} 
\usepackage[hyperpageref]{backref}

\numberwithin{equation}{section}
\numberwithin{figure}{section}

\usepackage{xargs}  %

\newtheorem{thm}{Theorem}[section]
\newtheorem{rem}{Remark}[section]

\newcommand{\commentout}[1]{{}} %

\newcommand{\bfa}{{\mbf a}}

\newcommand{\bfE}{{\mbf E}}

\newcommand{\bff}{{\mbf f}}

\newcommand{\bfH}{{\mbf H}}

\newcommand{\bfL}{{\mbf L}}

\newcommand{\bfp}{{\mbf p}}

\newcommand{\bft}{{\mbf t}}

\newcommand{\bfu}{{\mbf u}}

\newcommand{\bfv}{{\mbf v}}
\newcommand{\bfw}{{\mbf w}}

\newcommand{\bfx}{{\mbf x}}

\newcommand{\bfy}{{\mbf y}}

\newcommand{\bfxi}{\boldsymbol{\xi}}
\newcommand{\bfeta}{\boldsymbol{\eta}}

\newcommand{\vertiii}[1]{{\left\vert\kern-0.25ex\left\vert\kern-0.25ex\left\vert #1
    \right\vert\kern-0.25ex\right\vert\kern-0.25ex\right\vert}}
    \newcommand{\vertii}[1]{{\left\vert\kern-0.25ex\left\vert #1
    \right\vert\kern-0.25ex\right\vert}}
\newcommand{\verti}[1]{{\left\vert #1
    \right\vert}}

\graphicspath{{./figures/}}

\newcommand{\curl}{\operatorname{curl}}
\renewcommand{\div}{\operatorname{div}}

\newcommand{\bs}{\boldsymbol}
\newcommand{\mbf}{\boldsymbol}

\newcommand{\revision}[1]{\textcolor{black}{#1}}
\newcommand{\dd}{\,{\rm d}}

\begin{document}

\markboth{S. Cao, L. Chen \& R. Guo}{A VEM for 2D Maxwell interface problems}

%
\catchline{}{}{}{}{}
%

\title{A Virtual Finite Element Method for Two-Dimensional Maxwell Interface Problems with a Background Unfitted Mesh}

\author{Shuhao Cao}

\address{Department of Mathematics and Statistics, Washington University in St. Louis\\
 St. Louis, MO 63130\\
s.cao@wustl.edu}

\author{Long Chen}
\address{Department of Mathematics, University of California Irvine\\
 Irvine, CA 92697\\
chenlong@math.uci.edu}

\author{Ruchi Guo}
\address{Department of Mathematics, University of California Irvine\\
 Irvine, CA 92697\\
ruchig@uci.edu}

\maketitle


\begin{abstract}
A virtual element method (VEM) with the first order optimal convergence order is developed for solving two-dimensional Maxwell interface problems on a special class of polygonal meshes that are cut by the interface from a background unfitted mesh. A novel virtual space is introduced on a virtual triangulation of the polygonal mesh satisfying a maximum angle condition, which shares exactly the same degrees of freedom as the usual $\bfH(\curl)$-conforming virtual space. This new virtual space serves as the key to prove that the optimal error bounds of the VEM are independent of high aspect ratio of the possible anisotropic polygonal mesh near the interface.
\end{abstract}



\section{Introduction}
\label{sec:intro}
Maxwell interface problems widely appear in a large variety of science and engineering applications.
In this article, we propose a virtual element method (VEM) to solve a two-dimensional (2D) 
$\bfH(\text{curl})$-elliptic interface problem that originates from Maxwell equations.
One distinct advantage of the proposed method is its flexibility on the mesh generation to cater the  interface. The mesh for computation is obtained from a background unfitted mesh by cutting interface elements into triangles and quadrilaterals, and the optimal convergence order is guaranteed independent of the potential mesh anisotropy.

To describe the idea, we let $\Omega\subseteq\mathbb{R}^2$ be a bounded domain and let $\Gamma\subseteq \Omega$ be a closed smooth interface curve, as illustrated by the left plot in Figure \ref{fig:interf_dom}. The interface $\Gamma$ cuts $\Omega$ into two subdomains $\Omega^{\pm}$ occupied by media with different magnetic and electric properties. We consider the following $\bfH(\text{curl};\Omega)$-elliptic interface problem for the electric field $\bfu: \Omega \to \mathbb{R}^2$:
\begin{subequations}
\label{model}
\begin{align}
\label{inter_PDE}
\underline{ \text{curl}}~\alpha \, \text{curl}~\bfu + \beta \bfu = \bff \;\;\;\; & \text{in} \; \Omega = \Omega^-  \cup \Omega^+,
\end{align}
with $\bff\in\bfL^2(\Omega)$, subject to the Dirichlet boundary condition:
\begin{align}
\label{bc}
\bfu\cdot\bft = 0 \;\;\;\;\; & \text{on} \; \partial\Omega,
\end{align}
where the operator $\text{curl}$ is for vector functions $\bfv=(v_1,v_2)^{\intercal}$ such that $\text{curl}~\bfv=\partial_{x_1}v_2 - \partial_{x_2}v_1$,
while $\underline{\text{curl}}$ is for scalar functions $v$ such that $\underline{\text{curl}}~v = \left ( \partial_{x_2}v, - \partial_{x_1}v \right )^{\intercal}$ with ``$^{\intercal}$" denoting the transpose herein. The coefficients $\alpha=\alpha^{\pm}$ and $\beta=\beta^{\pm}$ in $\Omega^{\pm}$ are assumed to be positive piecewise constant functions of which the locations of the discontinuity align with one another. Moreover, we consider the following jump conditions at the interface $\Gamma$:
\begin{align}
[\bfu\cdot\bft]_{\Gamma} &:=  \bfu^+\cdot\bft -  \bfu^-\cdot\bft = 0,  \label{inter_jc_1} \\
[\alpha\, \text{curl}~\bfu]_{\Gamma} &:=  \alpha^- \text{curl}\,\bfu^+ - \alpha^+ \text{curl}\,\bfu^-   = 0,
\label{inter_jc_2}
\end{align}
\end{subequations}
where $\bs t$ denotes a tangential vector to $\Gamma$. The condition \eqref{inter_jc_1} is due to the tangential continuity of $\bs H(\curl)$ functions and \eqref{inter_jc_2} is from the fact that $H(\underline{\text{curl}};\Omega)$ is isomorphic to $H^1(\Omega)$ in 2D.
The interface model \eqref{model} arises from each time step in a stable time-marching scheme for the eddy current computation of Maxwell equations~\cite{2000AmmariBuffaNedelec,2000BeckHiptmairHoppeWohlmuth}.
In this model, $\alpha$ denotes the magnetic permeability and $\beta\sim \sigma/\triangle t$ 
is a scaling of the conductivity $\sigma$ by the time-marching step size $\triangle t$.

For Maxwell equations, $\bfH(\text{curl})$-conforming N\'ed\'elec finite element spaces are widely used~\cite{1999CiarletZou,2002Hiptmair,1992MonkSiam}. As for interface problems, the authors in Ref.~\refcite{2012HiptmairLiZou} analyze the standard finite element methods (FEMs) for $\bfH(\text{curl})$-elliptic equations. A semi-discrete analysis for Maxwell interface problems with low regularity is provided in Ref.~\refcite{2000Zhao}.
In addition, due to the potentially low regularity, there are many works focusing on developing a posteriori error estimators and adaptive FEM~\cite{2015CaiCao,2009ChenXiaoZhang,2016DuanQiuTanZheng}.

Numerical methods for solving interface problems based on unfitted meshes are attractive since they circumvent the burden of generating high-quality interface-fitted or interface-approximated meshes which may be time-consuming in three dimensions (3D) or for moving interface problems. There have been a lot of works in this field on solving $H^1$-elliptic interface problems, see e.g., Ref.~\refcite{2015BurmanClaus,2016GuoLin,1994LevequeLi,2015ChenwuXiao} and the reference therein. However, there are much fewer works on solving Maxwell interface problems. Typical examples include matched interface and boundary (MIB) formulation~\cite{2004ZhaoWei},
and non-matching mesh methods~\cite{2016CasagrandeHiptmairOstrowski,2016CasagrandeWinkelmannHiptmairOstrowski,2000ChenDuZou}.
Recently, a penalty method is developed requiring a higher regularity \cite{2020LiuZhangZhangZheng}.

A main difficulty for solving \revision{even non-interface} $\bfH(\curl)$ problems stems from the low regularity of the exact solution. \revision{In a practical setting with an $L^2$-integrable boundary condition, the geometry singularities result the functions in $\bfH(\curl)\cap \bfH(\div)$ having barely an $H^{1/2}$ regularity\cite{Costabel1990remark}, due to the low regularity of the pure Neumann problem for the potential function used to construct a Helmholtz decomposition.} 
In the context of the a priori error analysis, $\bfu\in\bfH^1(\curl ;\Omega) := \{\bs u\in \bs H^1(\Omega), \curl \bs u \in H^1(\Omega)\}$, the anticipated optimal convergence rate highly relies on the conformity of approximation spaces, due to the $\mathcal{O}(h^{1/2})$ approximation order on the boundaries of elements for functions in $\bfH^1(\curl ;\Omega)$; specifically, see Lemma 5.52 in Ref.~\refcite{2003Monk}. For example, when solving Maxwell equations by discontinuous Galerkin (DG) methods~\cite{2005HoustonPerugiaSchneebeli,2004HoustonPerugiaSchotzau,2005HoustonPerugiaDominik}, penalties are in general needed on boundary of elements due to the non-conformity of DG spaces, and the standard argument directly applying the trace inequalities may only yield suboptimal convergence rates. Instead, the approaches for the analysis of DG methods employ an $\bfH(\curl)$-conforming subspace of the broken DG space to overcome this issue.

This essential difficulty will be carried over to the development of an optimally convergent method on an unfitted mesh for Maxwell interface problems, since almost all methods on unfitted meshes use non-conforming spaces for approximation.
Together with a worsened regularity due to the presence of the interface \cite{Huang;Zou:2007Uniform,1999CostabelDaugeNicaise,BonitoGuermondLuddens2013Regularity}, this problem becomes distinctly challenging. 
\revision{Moreover, even if these non-conforming spaces have conforming subspaces, it is unclear whether these subspaces} have sufficient approximation capabilities such that the approaches in Ref.~\refcite{2005HoustonPerugiaSchneebeli,2004HoustonPerugiaSchotzau,2005HoustonPerugiaDominik} can be applied.
Indeed, Casagrande et al. show that using Nitsche's penalties on interface edges can only yield suboptimal convergence rates in both computation and analysis~\cite{2016CasagrandeHiptmairOstrowski,2016CasagrandeWinkelmannHiptmairOstrowski}. Recently, this issue was further explored numerically in Ref.~\refcite{2020RoppertSchoderTothKaltenbacher}. An alternative approach is to use immersed finite element methods in a Petrov-Galerkin formulation~\cite{2020GuoLinZou},
where the standard conforming N\'ed\'elec space is used as the test function space to remove the non-conformity errors. However, the resulted matrix may not be symmetric anymore, which could cause troubles for fast solvers \revision{and introduce the possibility of the loss of energy conservation}.

Motivated by the work in Ref.~\refcite{Chen;Wei;Wen:2017interface-fitted,2018CaoChen}, we believe that the virtual element method (VEM) provides a new direction to approximate Maxwell interface problems that can achieve an optimal convergence on (background) unfitted meshes. 
The VEM was first introduced in Ref.~\refcite{2013BeiraodaVeigaBrezziCangiani} to solve $H^1$-elliptic equations, where the $H^1$-virtual space consists of virtual shape functions that are solutions to local problems on elements with general polygonal shapes. 
The $\bfH(\curl)$-conforming virtual space was then introduced in Ref.~\refcite{2017VeigaBrezziDassiMarini,2016VeigaBrezziMarini} to solve magnetostatic problems. As one attractive feature, the underling virtual space for approximation is always conforming on an almost arbitrary polygonal mesh of the computation domain. It is our key motivation to use it for solving Maxwell interface problems on meshes that are generated from a background unfitted mesh. 
Different from Ref.~\refcite{2017VeigaBrezziDassiMarini,2016VeigaBrezziMarini} that use a mixed formulation~\cite{1989Kikuchi}, in this work we shall study the symmetric and positive definite $\bfH(\curl)$-elliptic equation \eqref{inter_PDE} as the model problem. 
Very recently, a similar VEM for Maxwell equations with the lowest order elements and on shape-regular meshes is analyzed in Ref.~\refcite{2021VeigaDassiMascotto}. Compared with Ref.~\refcite{2021VeigaDassiMascotto}, our results focus more on the discretization's robustness to the shape of elements, while the analysis only relies on mature simplicial finite element tools.

In our analysis, the key to achieve the optimal error bound regardless of element shapes is a novel virtual element space that shares exactly the same degrees of freedom of the one constructed in Ref.~\refcite{2017VeigaBrezziDassiMarini,2016VeigaBrezziMarini}. Thus, compared with the virtual space in Ref.~\refcite{2017VeigaBrezziDassiMarini,2016VeigaBrezziMarini}, this new space preserves all the necessary information for computation, including the same projection and curl values.
This space is constructed as a subspace of the standard N\'ed\'elec space on a further (virtual) triangulation of the polygonal mesh that satisfies a maximum angle condition~\cite{1999AcostaRicardo}.
Locally on each polygonal element, the new virtual functions can be also considered as discrete harmonic extensions according to the boundary conditions defined by the degrees of freedom, opposed to the continuous extensions of the usual virtual functions~\cite{2017VeigaBrezziDassiMarini,2016VeigaBrezziMarini}. A noteworthy advantage of using this new space is that we are able to establish local Poincar\'e-type inequalities and optimal approximation capabilities for a large class of polygonal-shaped elements, both of which are independent of element anisotropy. 
These estimates are crucial intermediate results toward the final optimal error bound.
A related work \cite{2009ChenXiaoZhang} constructs sub-meshes on interface elements of a background unfitted mesh to perform computation. One essential difference between our proposed method and Ref.~\refcite{2009ChenXiaoZhang} is that the virtual mesh and space are only used for analysis in our work, while the computation procedure is the same as the ordinary lowest-order edge VEM. The convergence is guaranteed independent of the element shape, provided that the background mesh is shape-regular.

This article consists of 5 additional sections. In the next section, we introduce a background unfitted mesh and \revision{a piecewise linear interface-approximated mesh} according to the interface geometry. In Section \ref{sec:VEM}, we describe virtual spaces and projection operators. In Section \ref{sec:VEM_err_eqn}, we present the numerical scheme and derive the error equation. In Section \ref{sec:IFE_disre}, we estimate the interpolation errors. In Section \ref{sec:convergence}, we show that the convergence is of optimal order. In the last section, some numerical examples are presented to verify the theoretical estimates.

\section{Preliminaries}

In this section, we first describe an unfitted background triangular mesh, and then locally partition a triangular interface element into \revision{a linear interface-approximated mesh} used for computation. We then introduce Sobolev spaces for $\boldsymbol{H}(\curl)$-interface problems. Although the triangular and quadrilateral shape is the focus of this work, we highlight that most key results are actually established and presented for more general polygonal element shapes. 

\begin{figure}[h]
\centering
   \includegraphics[width=1.6in]{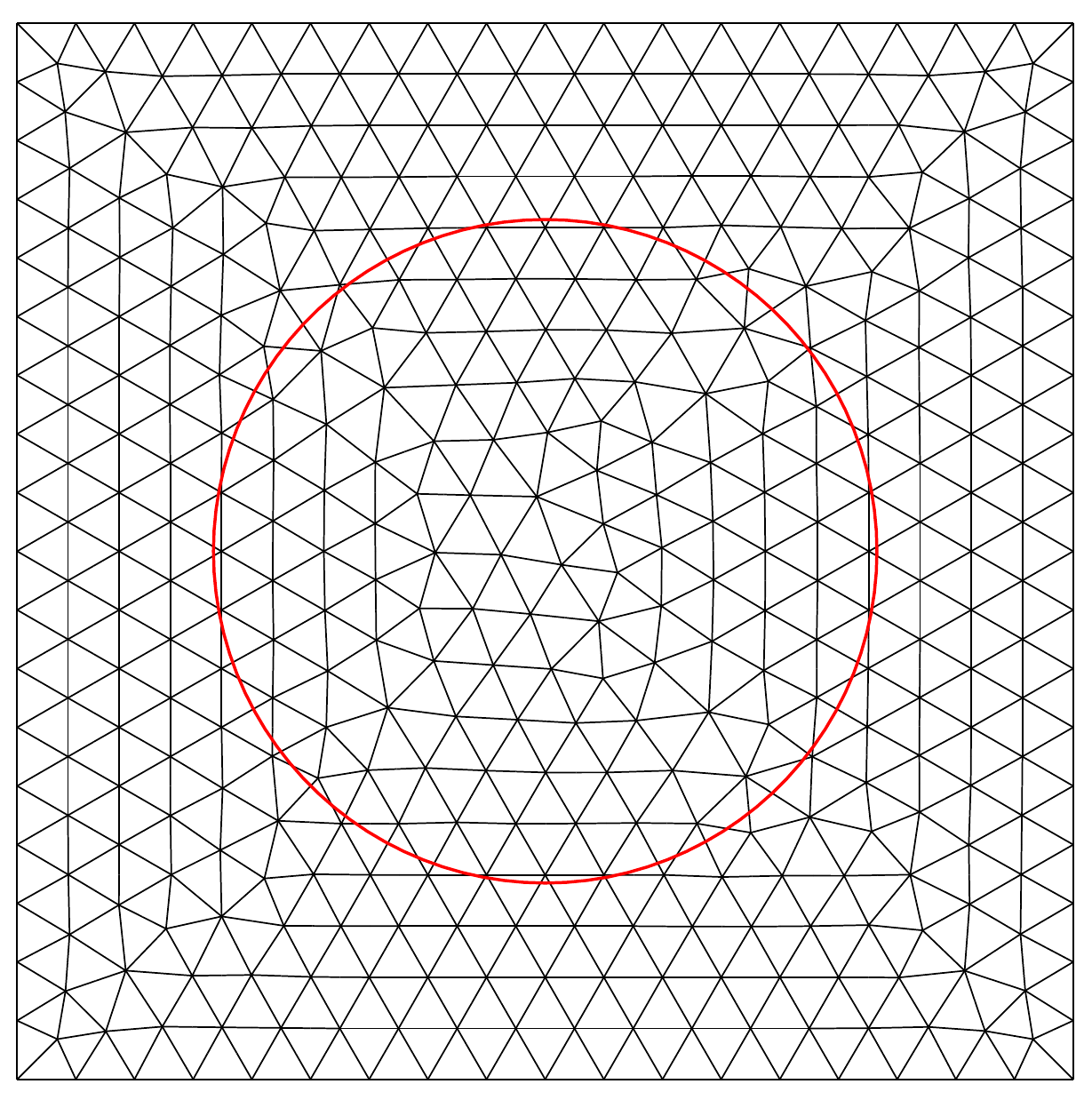}
\qquad \qquad \qquad
   \includegraphics[width=1.6in]{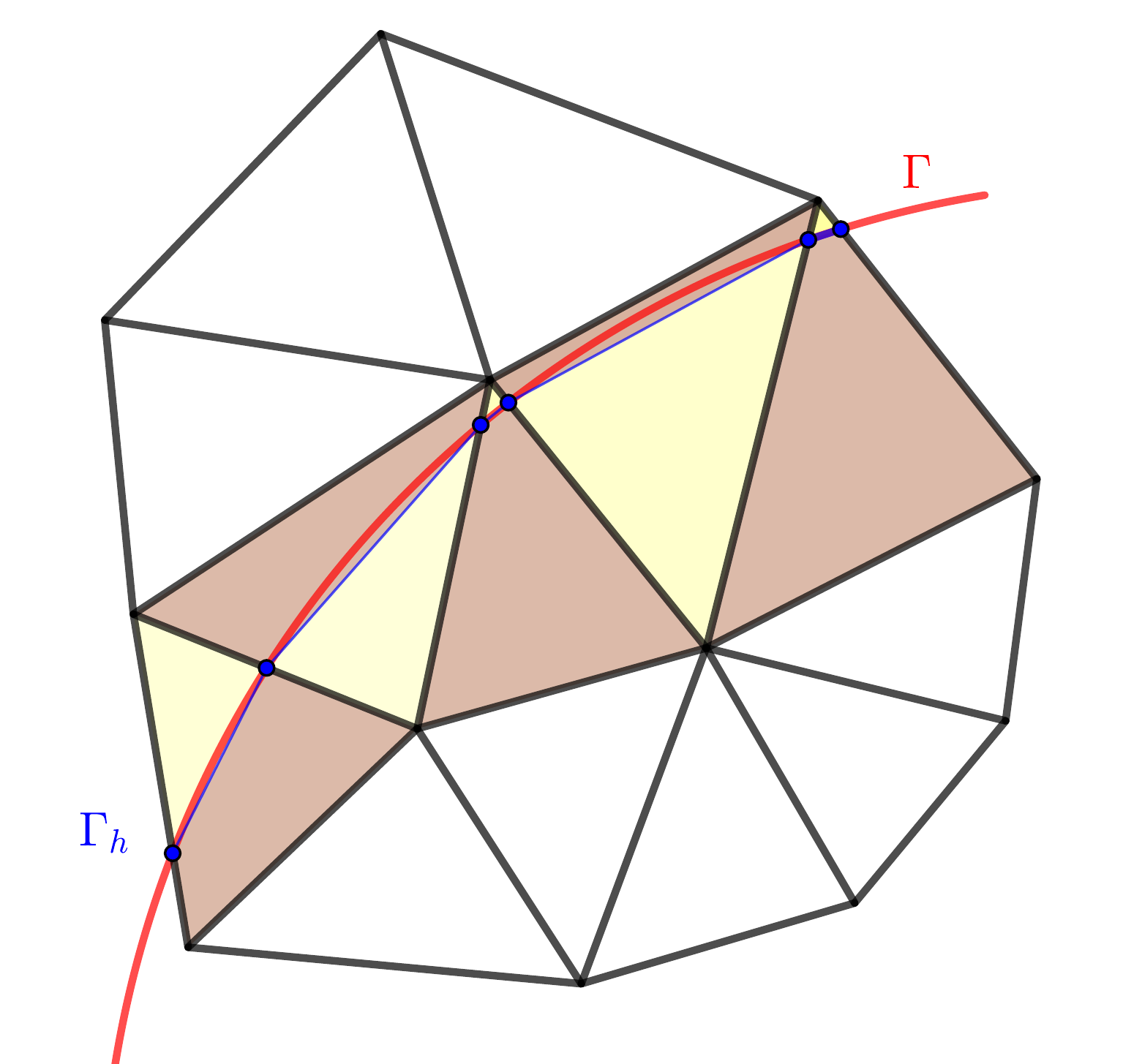}
  \caption{Left: a background unfitted mesh. Right: interface elements in the unfitted mesh are further partitioned into quadrilateral and triangular elements, shaded by brown and yellow colors respectively.}
  \label{fig:interf_dom}
\end{figure}

Consider an interface-independent shape-regular triangular mesh of the domain $\Omega$. Note that this mesh could be simply taken as a highly-structured mesh due to the interface independence. We shall call it a background mesh, and denoted it by $\mathcal{T}^B_h$. Another example of $\mathcal{T}^B_h$ is a uniform Cartesian grid which is widely used for unfitted mesh methods. 
The proposed analysis approach can be easily adapted to this grid.  If a triangular element in $\mathcal{T}^B_h$ intersects the interface, then it is called an interface element. The collection of interface elements is denoted as $\mathcal{T}^{Bi}_h$. The remaining elements are called non-interface elements. For this background mesh, we further make the following assumptions:
\begin{itemize}
  \item[(\textbf{A})]\label{asp:background} Each interface element intersects with $\Gamma$ at most two distinct points on two different edges.
 \item[(\textbf{B})] Each interface element does not intersect with the boundary of $\Omega$.
\end{itemize} 
By assumption (\textbf{A}), on an interface element $K\in\mathcal{T}^{Bi}_h$ we define $\Gamma^{K}_h$ as the line connecting the two intersection points. All these connected small segments, denoted as $\Gamma_h$, form a piecewise linear approximation of the true interface $\Gamma$. In addition, by the assumption (\textbf{A}), each triangular interface element $K$ is cut by $\Gamma^{K}_h$ into a triangular and a quadrilateral subelement from which \revision{a piecewise linear interface-approximated mesh can be generated. We denote this new mesh by $\mathcal{T}_h$}. The collections of quadrilateral and triangular elements in $\mathcal{T}_h$ resulted by the interface-cutting are denoted by $\mathcal{T}^q_h$ and $\mathcal{T}^t_h$, respectively. Those elements all have one edge aligning with the interface approximately, and thus are also called interface elements in $\mathcal{T}_h$. Clearly, there holds
$$
\overline{ \cup \{ K\in \mathcal{T}^t_h\cup \mathcal{T}^q_h \} } = \overline{ \cup \{ K\in \mathcal{T}^{Bi}_h \} }.
$$
Note that $\mathcal{T}_h$ and $\mathcal{T}^B_h$ are only different on the interface elements. Furthermore, an interface element $K$ is assumed to be cut into $K^-_h$ and $K^+_h$ by $\Gamma^{K}_h$. The mismatch portion, with $K^{\pm}:=K\cap\Omega^{\pm}$ cut from the original interface, is denoted by $K_{\rm int}$, indicated by the shaded region in the left plot of Figure \ref{fig:interf_elem_cut}.

In addition, for each interface edge $\Gamma^K_h$, we assume there is a shape regular triangle $B_h^K\subseteq \Omega$ with the base $\Gamma_h^K$ and a height $\mathcal{O}(h_K)$ which will be used to lift the trace on $\Gamma_h^K$ to the interior of domain $\Omega$. Therefore $B_h^K$ is not required to align with any elements in the mesh. We further assume the union of $B_h^K$'s for all interface edges have finite overlapping. Note that for the considered background regular triangular mesh, for each interface element $K$, this $B_h^K$ certainly exists and can be further shown to be contained in $K$.

\begin{figure}[h]
\centering
   \includegraphics[width=2in]{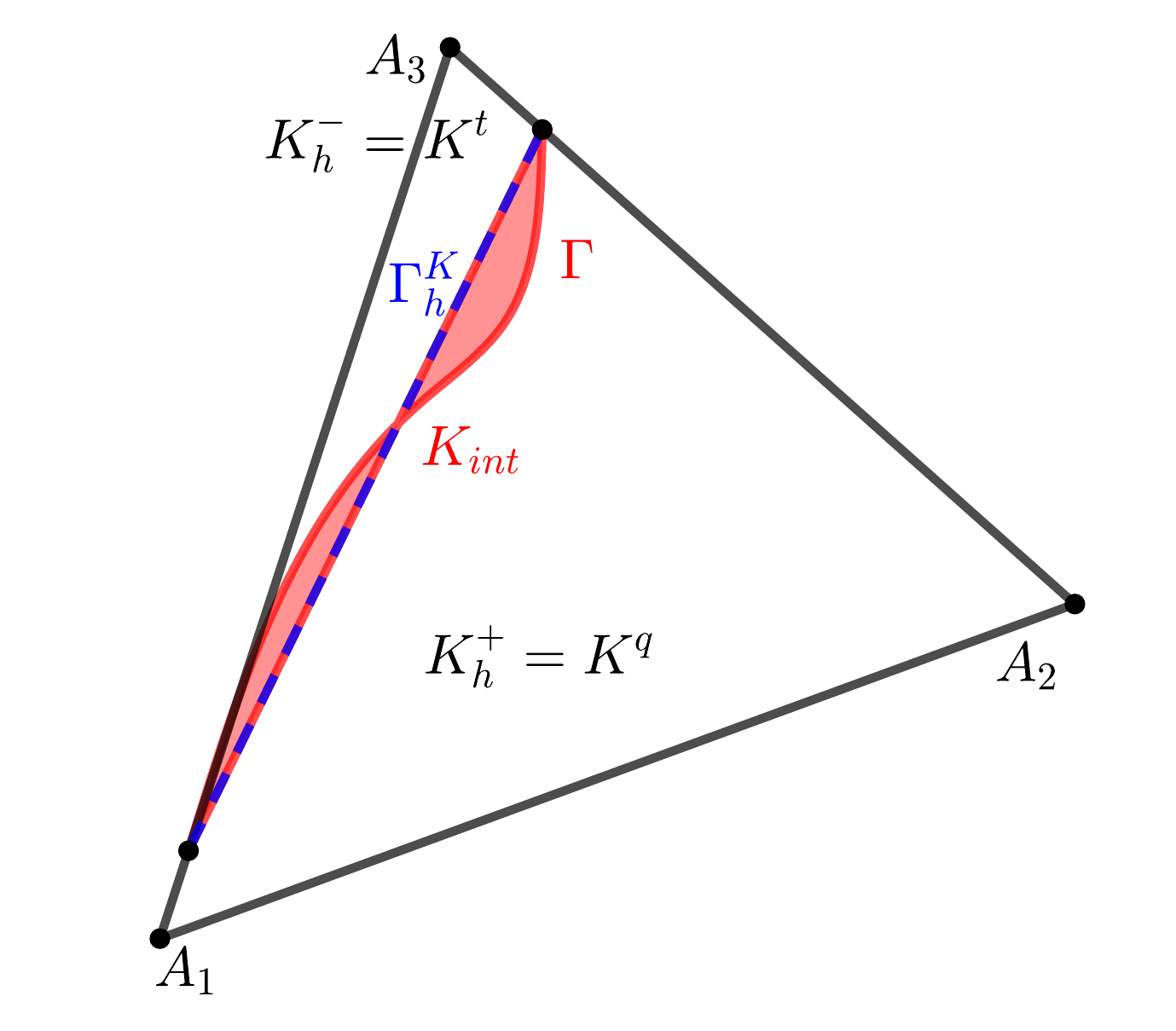}
   \quad
   \includegraphics[width=2in]{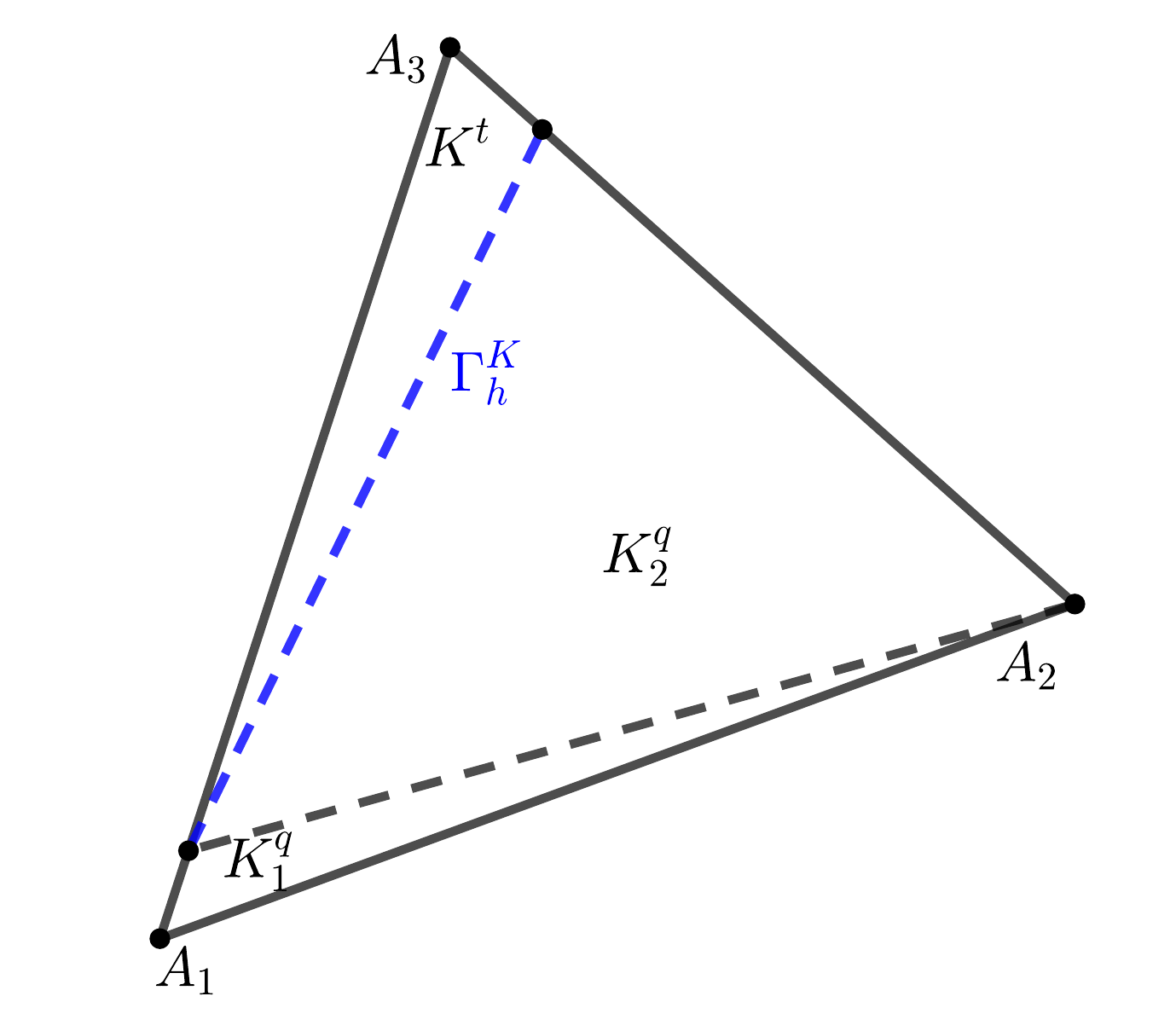}
  \caption{Left: an interface element $K\in\mathcal{T}^B_h$ is cut to a triangular element $K^t$ and a quadrilateral element $K^q$ of which both are in $\mathcal{T}_h$. Right: the quadrilateral element is further cut into two triangular elements $K^q_1$ and $K^q_2$.}
  \label{fig:interf_elem_cut}
\end{figure}

Moreover, we assume that the interface is well-resolved by the mesh, and it can be quantitatively described in terms of the following lemma~\cite{2016GuoLin}. 

\begin{lemma}
\label{lemma_interface_flat}
Suppose the mesh is sufficiently fine in the sense that there exists a certain threshold $h_0$ such that when $h<h_0$; on each interface element $K\in\mathcal{T}^{Bi}_h$, there exist a constant $C$ independent of the interface location inside $K$ and $h_K$ such that for every point $x\in\Gamma\cap K$ with its orthogonal projection $x^{\bot}$ onto $\Gamma^{K}_h$,
\begin{align}
   \emph{dist}(x,x^{\bot}) \le  C h_{K}^2. \label{lemma_interface_flat_eq1}  
\end{align}
\end{lemma}

The explicit dependence of $h_0$ on the curvature of the interface can be found in Ref.~\refcite{2016GuoLin}. An adaptively-generated background mesh can be found in Ref.~\refcite{Wei;Chen;Huang;Zheng:2014Adaptive} capturing large curvature of the interface. Now following the convention in Ref.~\refcite{2012HiptmairLiZou,2010LiMelenkWohlmuthZou}, we introduce the $\delta$-strip:
\begin{equation}
\label{delta_strip_1}
S_{\delta} := \{ x\in\Omega : \text{dist}(x,\Gamma)< \delta \}, ~~~ \text{and} ~~~ S^{\pm}_{\delta} := \{ x\in\Omega^{\pm} : \text{dist}(x,\Gamma)< \delta \}.
\end{equation}
By estimate \eqref{lemma_interface_flat_eq1}, we have 
\begin{equation}
\label{delta_strip_2}
 \cup\{ K_{\rm int}: K\in \mathcal{T}^{Bi}_h \} \subseteq S_{\delta}, ~~~ \delta\le C_{\Gamma}h^2
\end{equation}
with the constant $C_{\Gamma}$ only depending on the interface. Furthermore, we can control the $L^2$-norm in the $\delta$-strip by the lateral width of the strip~\cite{2012HiptmairLiZou,2010LiMelenkWohlmuthZou}, and thus obtain an order $\mathcal{O}(h)$ convergence when $\delta = \mathcal{O}(h^2)$.
\begin{lemma}[\revision{The proof in Lemma 3.4 and Lemma 2.1 in Ref.~\refcite{2010LiMelenkWohlmuthZou}}]
\label{lem_delta}
It \revision{holds true for} any $z\in H^1(\Omega^{\pm})$ that
\begin{equation}
\label{lem_delta_eq0}
\| z \|_{L^2(S^{\pm}_{\delta})} \le C \sqrt{\delta} \| z \|_{H^1(\Omega^{\pm})}.
\end{equation}
\end{lemma}

Next, we introduce some Sobolev spaces used throughout this article. For each subdomain $\omega\subseteq\Omega$, we let $H^s(\omega)$ and $\bfH^s(\omega)$, $s\ge0$, be the standard scalar and 2D vector Hilbert spaces on $\omega$, respectively. Specifically, $H^0(\omega)=L^2(\omega)$, and $\bfH^0(\omega) =\bfL^2(\omega)$. In addition, for $s\ge 0$, we let
\begin{equation}
\label{Hcurl_spa}
  \bfH^s(\text{curl};\omega) = \{ \bfv\in\bfH^s(\omega)~:~\text{curl}~\bfv \in H^s(\omega) \}.
  \end{equation}
Similarly, we introduce $\bfH^s(\text{div};\omega)$ as the counterpart of $\bfH^s(\text{curl};\omega)$ with the divergence operator. If $\omega\cap\Gamma\neq\emptyset$, then $\omega^{\pm}=\omega\cap\Omega^{\pm}$, and $\bfH^s(\text{curl};\omega^-\cup\omega^+)$ denotes the space of piecewise-defined functions in $\bfH^s(\text{curl};\omega^{\pm})$. For these spaces, we can define subspaces $H^s_0(\omega)$, $\bfH^s_0(\omega)$ with zero traces on $\partial \omega$, \revision{and a zero tangential trace for $\bfH^s_0(\text{curl};\omega)$}. Additionally, let $(\cdot,\cdot)_{\omega}$ be the standard $L^2$ inner product on $\omega$.

When the interface is smooth, the solution to the problem is expected to have an $\bfH^1(\text{curl};\Omega^{\pm})$ regularity \cite{1999CostabelDaugeNicaise,Huang;Zou:2007Uniform}. The fundamental $\bfH^1(\text{curl};\Omega)$-extension operator, established by Hiptmair, Li and Zou in Theorem 3.4 and Corollary 3.5 in Ref.~\refcite{2012HiptmairLiZou}, will be used in analysis.
\begin{thm}[Theorem 3.4 and Corollary 3.5 in Ref.~\refcite{2012HiptmairLiZou}]
\label{thm_ext}
There exist two bounded linear operators
\begin{equation}
\label{thm_ext_eq0}
\bfE^{\pm}_{\emph{curl}} ~:~ \bfH^1(\emph{curl};\Omega^{\pm})\rightarrow \bfH^1(\emph{curl};\Omega)
\end{equation}
such that for each $\bfu\in\bfH^1(\emph{curl};\Omega^{\pm})$:
\begin{itemize}
  \item[1.]  $\bfE^{\pm}_{\emph{curl}}\bfu = \bfu ~~ \text{a.e.} ~\text{in} ~ \Omega^{\pm}$.
  \item[2.] $\| \bfE^{\pm}_{\emph{curl}}\bfu \|_{\bfH^1(\emph{curl};\Omega)} \le C_E \| \bfu \|_{\bfH^1(\emph{curl};\Omega^{\pm})}$ with the constant $C_E$ only depending on $\Omega$ and $\Gamma$.
\end{itemize}
\end{thm}
Using these two special extension operators, we can define $\bfu^{\pm}_E := \bfE^{\pm}_{\text{curl}}\bfu^{\pm}$ which are the keys in the analysis later. Finally, throughout this article, for simplicity we shall use $\lesssim$ to denote a $\cdots\le C\cdots$ with a generic constant independent of the mesh size. the interface location relative to the background mesh, and the anisotropy of the element shape.

\section{Virtual Element Spaces}
\label{sec:VEM}

In this section, we shall introduce a virtual element space using the lowest order N\'ed\'elec element on a virtual triangulation which is obtained by refinement of the background mesh. On a polygon $P$, recall the solenoidal vector-valued polynomial that is used to define the first family simplicial N\'ed\'elec elements of the lowest order \cite{1980Nedelec} as follows:
\begin{equation}
\label{ned_spa}
    \mathcal{ND}_h(P) = \{ \bfa + b(x_2,-x_1)^{\intercal}: ~ \bfa\in\mathbb{R}^2, b\in\mathbb{R} \}.
\end{equation}

In the proposed method, triangular elements and quadrilateral elements in $\mathcal{T}_h$ are treated differently. To avoid confusion, in this section we shall usually use $K^t$ and $K^q$ to denote triangular and quadrilateral elements in $\mathcal{T}^t_h$ and $\mathcal{T}^q_h$, respectively, while $K$ denotes an interface element in the background mesh $\mathcal{T}^B_h$ or general elements in $\mathcal{T}_h$ if there is no need to distinguish their shape. In addition, for simplicity's sake, we always use $h_K$ as the diameter of elements $K$, $K^q$ and $K^t$. On a triangular element $K$, which can be either a non-interface element $K \in \mathcal{T}^B_h$ (or $\mathcal{T}_h$), or a triangular interface element $K^t$, \eqref{ned_spa} will be used directly to define the local space. Nevertheless, for $K^q$, we need to employ a virtual element space as follows. %

\subsection{Virtual edge element spaces}
Let us first discuss the definition on a general polygon $P$:
\begin{align}
\label{virtual_space_1}
\tilde{V}_h(P) = \{ \bfv_h \in \; &\bfH(\curl ;P)\cap\bfH(\text{div};P):  ~\nonumber \bfv_h\cdot\bft_e \in \mathbb{P}_0(e), ~\forall e\subset\partial P,\\
 &\text{div}(\bfv_h) = 0, ~ \curl (\bfv_h) \in \mathbb{P}_0(P) \}.
\end{align}
This is exactly the lowest degree variant of the local spaces introduced in Ref.~\refcite{2016VeigaBrezziMarini,2017VeigaBrezziDassiMarini}. Moreover, we can see $\mathcal{ND}_h(P) \subset \tilde{V}_h(P)$, \revision{where $\mathcal{ND}_h(P)$ is seen as a polynomial space not a local simplicial finite element space}. Similar to the ones for~\eqref{ned_spa} defined on triangles, the local degrees of freedom (DoFs) for~\eqref{virtual_space_1} are
\begin{equation}\label{vemdof}
\bfv_h\big|_{e}\cdot\bft_e, \quad e\subset\partial P.
\end{equation}
Any function in~\eqref{virtual_space_1} can be uniquely determined by DoFs \eqref{vemdof}, e.g., in Ref.~\refcite{2016VeigaBrezziMarini,2017VeigaBrezziDassiMarini}; see also  a rotated version with a curl-free constraint in Ref.~\refcite{Cao2021quadtree}. For the present situation, $\tilde{V}_h(K^q)$ ($P=K^q$) is the discretization space on $K^q$.

As the shape of elements $K^q$ could be very anisotropic, a robust norm equivalence and interpolation error estimate is hard to establish in $\tilde{V}_h(K^q)$.
To address this issue, we shall introduce an auxiliary triangulation of $\Omega$ (a virtual mesh), and construct an auxiliary $\bfH(\curl )$-conforming space associated with this mesh. Given an interface element $K\in \mathcal{T}^{Bi}_h$, with a quadrilateral subelement $K^q\in \mathcal{T}^q_h$, then the local auxiliary mesh is formed by a Delaunay triangulation of $K^{q}$: connecting the diagonal such that the sum of angles opposing to the diagonal is less than or equal to $\pi$; see the right plot in Figure \ref{fig:interf_elem_cut} for an illustration. Although it may contain anisotropic triangles, each triangle from this new partition satisfies the maximum angle condition (Lemma \ref{max_angle}), which is key to robust interpolation error estimates (see Section \ref{sec:proj-interp}). A similar result is proven in Ref.~\refcite{Chen;Wei;Wen:2017interface-fitted} for Cartesian grids.

\begin{lemma}
\label{max_angle}
Let $K$ be a shape regular triangle, i.e., there exist $0<\theta_{\min}\leq \theta_{\max}<\pi$ such that every angle $\theta$ in $K$ satisfies $\theta_{\min} \leq \theta \leq \theta_{\max}$, then every triangle in the auxiliary Delaunay triangulation on $K$ described above satisfies the maximum angle condition, i.e, every $\tilde{\theta}$ in the auxiliary triangulation is bounded above by $\tilde{\theta}_{\max}\leq \max\{\pi - \theta_{\min}, \theta_{\max}\}$.
\end{lemma}
\begin{proof}
Without loss of generality, we consider the triangle in the right plot of Figure \ref{fig:interf_elem_cut} for illustration where the left and right cutting points are $D$ and $E$, respectively. 
We shall bound the angles of three triangles: $K^t = \Delta DEA_3, K_1^q = \Delta A_1A_2D,$ and $K_2^q = \Delta DA_2E$. If the angle is one of (or part of) the angles of $\Delta A_1A_2A_3$, then it is bounded by $\theta_{\max}$. If the triangle contains one of the angles of $\Delta A_1A_2A_3$, as the sum of three angles is $\pi$, we conclude other angles are bounded by $\pi - \theta_{\min}$. Therefore, only the angles of the triangle $\Delta DA_2E$ need our attention.

Upon using the Delaunay property, $\angle DEA_2 +  \angle DA_1A_2 \leq \pi$, we get $\angle DEA_2 \leq \pi - \theta_{\min}$ and $\angle EDA_2\le \angle A_3DA_2 \le \pi - \angle A_1A_3A_2 \leq \pi - \theta_{\min}$. Thus, we have verified $\tilde{\theta}_{\max}\leq \max\{\pi - \theta_{\min}, \theta_{\max}\}$.

\end{proof}

The discussion for the special case on a quadrilateral $K^q$ is postponed to Section \ref{sec:IFE_disre}, and the results on a general polygon $P$ are the focus in the rest of this section. To be able to establish a robust analysis of the approximation capabilities of $V_h(P)$ below and the stability of the discretization, we introduce the following assumptions on the polygon $P$:
\begin{itemize}
\item[\revision{(P0)}] \label{asp:polygon} 
\revision{$P$ is assumed to admit a triangulation $\mathcal T_h(P)$ with no additional interior vertices added, i.e., the collection of edges $\mathcal{E}_h(P)$ in $\mathcal{T}_h(P)$ are solely formed by the vertices on $\partial P$;}
 \item[(P1)] the maximum angle condition holds uniformly for triangles in $\mathcal T_h(P)$;
 \item[(P2)] \label{asp:NSIE} no-short-interior-edge condition: $h_P\lesssim |e|$ for every interior edge $e$;
 \item[(P3)] star-convexity: there exists $\bfx:=(\bar{x}_1, \bar{x}_2) \in P$ such that $\overline{\bfx \bfy}\subseteq P$, $\forall \bfy \in \partial P$.
\end{itemize}
 We then define an auxiliary VEM space
\begin{equation}
\label{auxi_v_space_1_polygon}
\begin{aligned}
V_h(P) = \{ \bfv_h \in \bfH(\curl ;P): ~ & \bfv_h|_{K} \in \mathcal{ND}_h(K), \forall K\in \mathcal T_h(P), \\
& \curl \bfv_h \in \mathbb{P}_0(P)\}.
\end{aligned}
\end{equation}
Indeed we can show the VEM space defined by~\eqref{auxi_v_space_1_polygon} shares the same DoFs of~\eqref{vemdof}.
\begin{lemma}
\label{lem_iso_mor}
Let $P$ be a simple polygon \revision{satisfying assumption (P0)}, then the DoFs $\{\bfv_h\cdot\bft_e, e\subset\partial P\}$ are unisolvent on the space $V_h(P)$.
\end{lemma}
\begin{proof}
Let us consider the following problem: for any given boundary conditions $\bfv_h\cdot\bft_e$, $e\subset\partial P$, find $(\bfv_h,\lambda_h)\in \mathcal{ND}_{h}(\mathcal{T}_h(P))\times S_{h,0}(\mathcal{T}_h(P))$ satisfying
\begin{equation}
\label{harmon_ext_eq_1}
\left \{
\begin{aligned}
     (\curl \bs v_h, \curl \bfw_h )_P + (\bfw_h ,\nabla \lambda_h)_P & = 0, ~~~~ \forall \bfw_h \in \mathcal{ND}_{h,0}(\mathcal{T}_h(P)),
 \\
     (\bfv_h,\nabla p_h)_P & =0, ~~~~ \forall p_h \in S_{h,0}(\mathcal{T}_h(P)).
\end{aligned}
\right.
\end{equation}
In this problem, $S_{h,0}(\mathcal{T}_h(P))$ is the piecewise linear Lagrange finite element space with zero trace on $\partial P$. \revision{Note that this is a well-posed problem as the shape functions associated with the boundary DoFs can be moved to the right-hand side and, thus, it reduces to a system only about internal unknowns.} Now, $S_{h,0}(\mathcal{T}_h(P))$ is a trivial space, since there is no internal vertex. As such, \eqref{harmon_ext_eq_1} reduces to
\begin{equation}\label{eq:curlcrul}
(\curl \bs v_h, \curl \bfw_h )_P = 0, ~~~~ \forall \bfw_h \in \mathcal{ND}_{h,0}(\mathcal{T}_h(P)).
\end{equation}
Denoting the jump on an edge $e$ similar to those in \eqref{inter_jc_1}--\eqref{inter_jc_2}, the fact that $\curl \bfv_h$ is a piecewise constant on $\mathcal{T}_h (P)$ and an integration by parts show
\begin{equation}
\label{eq:curlcrul_2}
\sum_{e\in\mathcal{E}_h(P)} [\curl  \bfv_h]_e \int_e \bfw_h\cdot \bft_e \dd s =0, ~~~~ \forall \bfw_h \in \mathcal{ND}_{h,0}(\mathcal{T}_h(P)).
\end{equation}
Therefore, $\curl \bfv_h$ must be a single constant on all elements in $\mathcal{T}_h(P)$. Namely, the solution space of \eqref{harmon_ext_eq_1} is $V_h(P)$ in \eqref{auxi_v_space_1_polygon}. The unisolvence follows naturally from the homogeneous boundary condition yielding the zero solution.
\end{proof}

\begin{rem}
\label{harmon_ext}
Lemma \ref{lem_iso_mor} basically shows that $\bs v_h\in V_h(P)$ satisfies $\div_h \bs v_h = 0$ together with \eqref{eq:curlcrul}, where $\div_h$ is \revision{the weak divergence operator defined as the $L^2$-adjoint of $\nabla \!: S_{h,0}(\mathcal{T}_h(P)) \to \mathcal{ND}_{h,0}(\mathcal{T}_h(P))$}.
Namely, the functions in the VEM space~\eqref{auxi_v_space_1} \revision{can be seen as discrete harmonic extensions} of the boundary conditions $\bfv_h\cdot\bft$ on $\partial P$. Meanwhile, the functions in the original virtual space~\eqref{virtual_space_1} \revision{are continuous extensions with constraint $\div \bs v_h = 0$. Therefore, functions in $\tilde V_h(P)$ may not be polynomials, while the functions in $V_h(P)$ are piecewise vector-valued polynomials}, for which the error estimates are relatively easy to establish.
\end{rem}

By Lemma \ref{max_angle}, it is clear that $K^q=K^q_1\cup K^q_2$ induces such a triangulation satisfying \hyperref[asp:polygon]{(P0)--(P3)}. Meanwhile, we would like to remark that the aforementioned setting and the forthcoming analysis in this paper can be easily extended to the case when $\mathcal{T}^B_h$ is a uniform Cartesian grid, on which the interface elements consist either trapezoids, or triangle-pentagon satisfying \hyperref[asp:polygon]{(P0)--(P3)}.
In the subsequent analysis involving $K^q$, the space ${V}_h(K^q)$ ($P=K^q$) is replacing
$\tilde{V}_h(K^q)$
\begin{equation}
\label{auxi_v_space_1}
\begin{aligned}
V_h(K^q) = \{ \bfv_h \in \bfH(\curl ;K^q): ~ & \bfv_h|_{K^q_i} \in \mathcal{ND}_h(K^q_i), ~ i=1,2, \\
& \curl (\bfv_h) \in \mathbb{P}_0(K^q)\}.
\end{aligned}
\end{equation}
All these triangles on interface elements form a triangulation resolving the interface,
and the global $\bfH(\curl )$-conforming space is defined as
\begin{equation}
\begin{split}
\label{auxi_v_space_3}
V_h = \{ \bfv_h \in \bfH_0(\curl ;\Omega) :~ &\bfv_h \in \mathcal{ND}_h(K) ~ \text{on} ~K \notin \mathcal{T}^q_h,\\
& \text{and} ~ \bfv_h \in V_h(K) ~ \text{on} ~ K\in\mathcal{T}^q_h \}.
\end{split}
\end{equation}
As we assume the background mesh $\mathcal T_h^{B}$ is shape regular, the maximum angle condition holds uniformly for the auxiliary mesh $\mathcal T_h$ by Lemma \ref{max_angle}.

\subsection{Projection and interpolation operators}
\label{sec:proj-interp}
For a general polygon $P$ and $\bfv_h \in V_h(P)$, the constant $\curl \bs v_h|_P$ can be computed by
\begin{equation}
\label{curl_vt}
\curl \bfv_h = \frac{1}{|P|}\int_{P} \curl \bfv_h \dd x = \frac{1}{|P|}\int_{\partial P} \bfv_h\cdot\bft \dd s.
\end{equation}
Thus, with $\bs v_h\cdot \bs t$ known as DoFs, and $\curl \bs v_h$ obtained as above, an $L^2$-projection of $\bs v_h$ can be computed
\cite{2016VeigaBrezziMarini,2017VeigaBrezziDassiMarini}.
On any elements or edges $\omega\subseteq \Omega$ we define the local $L^2$ projection $\Pi_{\omega}: L^2(\omega) \rightarrow [\mathbb{P}_0(\omega)]^2$ such that
\begin{equation}
\label{l2projection}
(\Pi_{\omega} \bfv_h, \bfp )_{\omega} = ( \bfv_h, \bfp )_{\omega}, ~~~ \forall \bfp \in  [\mathbb{P_0}(\omega)]^2
\end{equation}
which is indeed computable according to the DoFs of~\eqref{virtual_space_1} (see Remark 3 in Ref.~\refcite{2017VeigaBrezziDassiMarini}). For readers' sake, we recall the procedure here: for each $\bfp = (p_1, p_2)^{\intercal}\in [\mathbb{P}_0(P)]^2$, there exists $\phi_h= -p_2 (x_1 - \bar x_1) + p_1 (x_2 -\bar x_2) \in \mathbb{P}_1(P)$, such that $\underline{\curl}\, \,\phi_h = \bfp$, where $(\bar{x}_1, \bar{x}_2)$ is the point in the star-convexity assumption \hyperref[asp:polygon]{(P3)},
Therefore
\begin{equation}
\begin{split}
\label{compute_proj}
(\Pi_{K^q} \bfv_h, \bfp )_{P} & = (\bfv_h, \bfp )_{P} =  ( \bfv_h,  \underline{\curl}\,  \phi_h )_{P} \\
& = (\curl  \bfv_h, \phi_h)_{P} - (\bfv_h\cdot\bft, \phi_h)_{\partial P}.
\end{split}
\end{equation}
As $\bs v_h\cdot \bs t_e$ is given as a DoF on each $e$, and $\curl \bs v_h$ is constant, we get
\begin{equation}
\label{proj_formula}
\Pi_{P} \bfv_h =|P|^{-1} \big((\bfv_h\cdot\bft, \bar{x}_2 - x_2)_{\partial P},
-(\bfv_h\cdot\bft, \bar{x}_1 - x_1)_{\partial P} \big)^{\intercal},
\end{equation}
in which the integration on $\partial P$ is with respect to $\dd s(x_1,x_2)$.

Due to the DoFs being imposed on edges, we can define the interpolation
\begin{equation}
\label{interp_2}
I_{P}: \bfH^1(\curl ;P)\rightarrow V_h(P),\quad \int_{e} I_{P}\bfu\cdot\bft \dd s = \int_e \bfu\cdot\bft \dd s, \quad  \forall e \subset \partial P.
\end{equation}
We note that if $P$ is a triangle, $I_P$ reduces exactly to the usual edge interpolation operator, and the special one is for other general polygons such as quadrilateral elements $K^q$ where shape functions are from the virtual space $V_h(P)$ in \eqref{auxi_v_space_1}. Using integration by parts,  we get
$$
\int_{P}\curl I_P \bs u \dd x = \int_{P}\curl \bs u \dd x.
$$
Namely, $\curl I_P \bs u$ is the $L^2$-projection of $\curl \bs u$ to the space of constants, \revision{in accordance with the commutative diagram between the continuous and discrete de Rham complexes thanks to \eqref{interp_2}}.

Moreover, the interpolation $I_{P}:\tilde{V}_h(P)\rightarrow V_h(P)$ serves as a bijective mapping which also preserves curl values. The $L^2$ projectors in $\left[ \mathbb{P}_0(P) \right]^2$ are also the same, as $\tilde{V}_h(P)$ and $V_h(P)$ share the same DoFs. For the considered mesh $\mathcal{T}_h$, taking $P=K\in\mathcal{T}_h$, we have $\tilde{V}_h$ and $V_h$ lead to the same numerical scheme, yet the analysis based on $V_h$ can exploit more existing tools built for simplicial finite elements.

Finally, a global interpolant $\bfu_I$ is formed by gluing these local interpolations together, for which certain modification must be introduced on the interface edges forming $\Gamma_h$ (see Section \ref{subsec:interp_def}).

In the rest of this section, we present some estimates which show the convenience in analysis of opting for the space $V_h(P)$. For a triangle with vertices $\bs a_i$, let $\theta_i$ be the angle at vertex $\bs a_i$ and $e_i$ be the edge opposite to $\bs a_i$, for $i=1,2,3$.

\begin{lemma}\label{lem_norm_equiv1_eq0}
The following identity \revision{holds true} for any linear $\phi_h$ on a triangle $T$:
\begin{equation}
\| \nabla \phi_h \|_{L^2(T)}^2 = R_T\sum_{i=1}^3\cos \theta_i \| \nabla \phi_h \cdot \bs t_i\|_{L^2(e_i)}^2,
\end{equation}
where $R_T$ is the circumradius of $T$ and $\bs t_i$ is a unit tangential vector of $e_i$.
\end{lemma}
\begin{proof}
Denote by $\phi_i := \phi_h(\bs a_i)$ for $i=1,2,3$. The cotangent formula \cite{Duffin1959Distributed,Strang;Fix:1973analysis} reads
$$
\| \nabla \phi_h \|_{L^2(T)}^2 = \frac{1}{2}\sum_{i=1}^3 \cot \theta_i ( \phi_{i-1} - \phi_{i+1})^2.
$$
Then the law of sines and $|\nabla \phi_h \cdot \bs t_i|^2 = (\phi_{i-1} - \phi_{i+1})^2/|e_i|^2$ imply the result.
\end{proof}

We now prove the following Poincar\'e-type inequality which is one of the keys for the analysis on anisotropic meshes.
\begin{lemma}\label{lem_Poincare}
Let $P$ be a simple polygon satisfying \hyperref[asp:polygon]{(P0)--(P3)}, then $\forall \bfv_h \in V_h(P)$,
\begin{equation}
\label{lem_norm_equiv2_eq0}
\| \bfv_h \|_{L^2(P)} \le  \left( \frac{ \cot(\theta_{\max}) C(N_P)  }{2\lambda} \right)^{1/2} (  h^{1/2}_{P} \|  \bfv_h\cdot\bft \|_{L^2(\partial P)} + h_{P} \| \emph{curl}\, \bfv_h \|_{L^2(P)}).
\end{equation}
where $\theta_{\max}$ is the maximum angle of the triangles in $\mathcal{T}_h(P)$, $\lambda = \min_e h_e/h_P$ for the interior edge $e$ in $\mathcal{T}_h(P)$, and $C(N_P)$ is an integer only depending on the number of vertices of $P$ denoted by $N_P$.
\end{lemma}
\begin{proof}
 Define an auxiliary function
$$
\bfw_h = \frac{ \curl  \bfv_h }{2} \left[\begin{array}{c} -(x_2-\bar x_2) \\ x_1 - \bar x_1 \end{array}\right],
$$
where $(\bar{x}_1, \bar{x}_2)$ is the point in the star-convexity condition in
\hyperref[asp:polygon]{(P3)}.
It is clearly that
\begin{align}
\label{lem_norm_equiv2_eq1}
\| \bfw_h \|_{L^2(P)} &\lesssim h_{P} \| \curl \bfv_h \|_{L^2(P)}.
\end{align}
In addition, for every edge $e\subset \partial P$, $(-(x_2-\bar x_2), x_1 - \bar x_1)^{\intercal}\cdot\bft|_e$ yields the height $l_e$ of $e$ in the triangle formed by $e$ and $(\bar{x}_1, \bar{x}_2)$,
thus we have $|\bfw_h\cdot\bft_e| = l_e |\curl  \bfv_h| /2$. Together with the star-convexity condition, we have
\begin{align}
\label{lem_norm_equiv2_eq2}
\| \bs w_h\cdot \bs t\|_{L^2(e)} & = \frac{h^{1/2}_el_e}{2} |\curl  \bfv_h|  \le h_{P}^{1/2} \| \curl \bs v_h\|_{L^2(P)}.
\end{align}
Furthermore, we note that $\curl (\bfw_h - \bfv_h) = 0$, then by a standard argument of the conforming exact sequence, there exists a continuous piecewise linear finite element function $\phi_h$ such that $\bfv_h - \bfw_h = \nabla \phi_h$. Applying Lemma \ref{lem_norm_equiv1_eq0}, %
we get the estimate
\begin{equation}\label{eq:firststep}
\| \bfv_h - \bfw_h \|_{L^2(P)}^2 \le \frac{1}{2} h_P \cot(\theta_{\max}) \sum_{e\in \mathcal E_h(P)} \| (\bfv_h - \bfw_h)\cdot \bs t_e \|_{e}^2.
\end{equation}
We then control the norm contribution from an interior edge $e$. Since $P$ is simply connected, any interior edge $e$ divides $P$ into two parts. Choose the part with less boundary edges and denote it by $P_e$. Note that $\int_{\partial P_e}\nabla \phi_h\cdot \bs t \dd s = 0$, consequently by $\nabla \phi_h\cdot \bs t$ being a constant on each edge on $\partial P_e$, we have an identity decomposing $\partial P_e = (\partial P_e  \cap \partial P) \cup e$,
\[
h_e \nabla \phi_h\cdot \bs t_e + \sum_{e_i\subset \partial P\cap \partial P_e} h_{e_i}\nabla \phi_h\cdot \bs t_{e_i} = 0,
\]
and thus
$$
\| \nabla \phi_h\cdot \bs t  \|_{L^2(e)}
\leq \sum_{e_i\subset \partial P\cap \partial P_e} \left (\frac{h_{e_i}}{h_{e}} \right )^{1/2} \|\nabla \phi_h\cdot \bs t_{e_i}\|_{L^2(e_i)}.
$$
Then from \eqref{eq:firststep}, we can get %
$$
\| \bfv_h - \bfw_h \|_{L^2(P)}^2
\le \frac{ \cot(\theta_{\max}) C(N_P)  }{2\lambda} h_P \sum_{e\subset \partial P} \left (\| \bfv_h \cdot \bs t_e \|_{L^2(e)}^2 +  \| \bfw_h \cdot \bs t_e \|_{L^2(e)}^2\right ).
$$
Finally, the desired estimate \eqref{lem_norm_equiv2_eq0} follows from the triangle inequality and estimates \eqref{lem_norm_equiv2_eq1}-\eqref{lem_norm_equiv2_eq2}.
\end{proof}


\section{A VEM Scheme and an Error Bound}
\label{sec:VEM_err_eqn}

In this section, we describe the proposed virtual element formulation and derive an error bound.
We start with the standard weak formulation: find $\bfu\in \mathbf{ \bfH}_0(\curl ,\Omega)$ such that
\begin{equation}
\label{weak_form_1}
a(\bfu,\bfv): = (\alpha \, \curl \bfu,\curl \bfv)_{\Omega} + (\beta \, \bfu, \bfv)_{\Omega}  = (\bff,\bfv)_{\Omega}, ~~~ \forall \bfv\in \mathbf{ H}_0(\curl ,\Omega).
\end{equation}

\subsection{A Galerkin method}
We emphasize that the local ``virtual'' element space~\eqref{auxi_v_space_1} and the global one~\eqref{auxi_v_space_3} is right away a computable space, readily used for the discretization, unlike~\eqref{virtual_space_1}. The DoF on the diagonal edge can be determined by solving~\eqref{eq:curlcrul} explicitly, and a set of modified harmonic bases on boundary edges can be obtained and used in computation. As a result, the standard Galerkin formulation is computable without referring to the VEM framework of a projection-stabilization split: find $\bfu_h\in V_h$ such that
\begin{equation}
\label{auxi_scheme}
(\alpha_h \, \curl  \bfu_h, \curl \bfv_h)_{\Omega} + (\beta_h \, \bfu_h, \bfv_h)_{\Omega} = (\bff, \bfv_h)_{\Omega}, ~~~ \bfv_h\in V_h,
\end{equation}
where $\alpha_h$ and $\beta_h$ are the modification of $\alpha$ and $\beta$ according to the linearly approximated interface $\Gamma_h$. No projection operator is required since all the shape functions are computable.

However, this approach will introduce an extra partition which becomes inefficient especially in 3D. Instead, we shall treat henceforth the interface part of $\mathcal{T}_h$ as a virtual mesh only appearing in analysis not computation, whereas this associates the meaning of ``virtual'' in $V_h$. Its approximation capabilities will be discussed in Section \ref{sec:est-max-angle} based on the maximum angle condition.

\subsection{A VEM scheme}
Using the $L^2$-projection \eqref{l2projection}, we define a bilinear form
\begin{equation}
\label{VEM1}
a_h(\bfu, \bfv) :=  (\alpha_h \, \curl \bfu, \curl \bfv )_{\Omega} + (\beta_h \, \Pi_h \bfu, \Pi_h \bfv)_{\Omega} + \sum_{K\in\mathcal{T}^{Bi}_h} S_{K}(\bfu, \bfv)
\end{equation}
where the operator $\Pi_h$ is taken as $\Pi_{K^q}$ if $K=K^q\in \mathcal{T}^q_h$, and the identity operator otherwise. The stabilization $S_{K}(\bfu_h, \bfv_h)$ is defined element-wisely only on $K^q\in\mathcal{T}^q_h$, i.e., the quadrilateral subelements of the interface elements $K\in\mathcal{T}^{Bi}_h$:
\begin{equation}
\label{VEM2}
S_{K}(\bfu, \bfv) := \gamma_{K} h_{K} \left( \beta_h(\bfu - \Pi_{K^q}\bfu)\cdot\bft, (\bs v - \Pi_{K^q}\bs v)\cdot\bft \right)_{\partial K^q}
\end{equation}
with a parameter $\gamma_K$ independent of the mesh size and specified later. Note that the motivation of this stabilization term comes from the approximation of $  (\beta_h(\bs u_h - \Pi_h \bs u_h), \bs v_h - \Pi_h \bs v_h)_{K^q}$, and thus suggests the scaling $h_K$ in~\eqref{VEM2}.

At last, the proposed VEM discretization is to find $\bfu_h \in V_h$ such that
\begin{equation}
\label{VEM3}
a_h(\bfu_h, \bfv_h)  = (\bff, \Pi_h \bfv_h)_{\Omega}, ~~~ \forall \bfv_h\in V_h.
\end{equation}
\revision{\begin{remark}
\label{rem_scaling}
It is highlighted that the stabilization in \eqref{VEM2} employs an $h_K$ scaling, instead of the $h^{-1}_K$ weighted penalty widely used in DG-type, hybrid, or nonconforming methods \cite{2016CasagrandeWinkelmannHiptmairOstrowski,2016CasagrandeHiptmairOstrowski,2005HoustonPerugiaSchneebeli,2004HoustonPerugiaSchotzau,AwanouFabienGuzmanStern2020Hybridization,MuWangYeZhang2015weak}. This scaling is one of the keys for optimal convergence if the solution only has $\bfH^1(\curl;\Omega)$ regularity. The regularity of the traces of the underlying Sobolev spaces on element boundaries plays an important role in understanding this scaling, e.g., see the regularity-dependent penalty scaling in Ref.~\refcite{BrennerLiSung2007Locally}. When stabilization is imposed using a weighted $L^2$-inner product, the $H^{1/2}$ trace of $H^1$ functions suggests the $h^{-1/2}$ scaling. While in 2D, the $H^{-1/2}$ trace of $\bfH(\curl)$ functions suggests the $h^{1/2}$ scaling. Suboptimal convergence may occur if the order of scaling does not match the regularity. 
\end{remark}}

\subsection{An error bound}
As mentioned in Section \ref{sec:VEM}, some elements could be extremely anisotropic, and the commonly used norm equivalence in the VEM framework may not be applicable. Following the approach in Ref.~\refcite{2018CaoChen}, we shall work on an induced norm on $V_h$ by the bilinear form in~\eqref{VEM1} (Lemma \ref{lem_energy_norm}) which is weaker than the original graph norm:
\begin{equation}
\begin{split}
\label{energy_norm}
\vertiii{\bfv_h}^2_h :& = \| \alpha^{1/2}_h \curl \bfv_h \|^2_{L^2(\Omega)} + \| \beta^{1/2}_h \Pi_h \bfv_h \|^2_{L^2(\Omega)} \\
& + \sum_{K\in\mathcal{T}^{Bi}_h} h_K \| (\bfv_h - \Pi_{K^q}\bfv_h)\cdot\bft \|^2_{L^2(\partial K^q)}.
\end{split}
\end{equation}
\revision{
\begin{lemma}
\label{lem_norm_quiv}
Let $P$ be a simple polygon satisfying \hyperref[asp:polygon]{(P0)--(P3)}, then the following coercivity holds true: 
$$
\| \bfv_h \|_{\bfH(\curl;P)} \lesssim \| \curl \bfv_h \|_{L^2(P)} + \| \Pi_P\bfv_h \|_{L^2(P)} + h^{1/2}_P \| (\bfv_h - \Pi_P\bfv_h)\cdot\bft \|_{L^2(\partial P)}.
$$
\end{lemma}
\begin{proof}
It suffices to bound $\| \bfv_h \|_{L^2(P)}$. First, the triangle inequality implies 
$$
\| \bfv_h \|_{L^2(P)} \lesssim \| \bfv_h - \Pi_P\bfv_h  \|_{L^2(P)} + \| \Pi_P\bfv_h  \|_{L^2(P)}.
$$
To bound $\| \Pi_{P} \bfv_h - \bfv_h \|_{L^2(P)}$, using Lemma \ref{lem_Poincare}
yields the following estimate
\begin{equation}
\label{lem_norm_quiv_eq1}
\| \Pi_{P} \bfv_h - \bfv_h \|_{L^2(P)} \lesssim h^{1/2}_K \| (\Pi_{P} \bfv_h - \bfv_h)\cdot\bft \|_{L^2(P)} + h_K \| \curl  \bfv_h \|_{L^2(P)}
\end{equation}
which finishes the proof.
\end{proof}
We highlight that the Poincar\'e inequality in Lemma \ref{lem_Poincare} is the key to obtain the robust coercivity on highly anisotropic polygonal elements.}
\begin{lemma}
\label{lem_energy_norm}
$\vertiii{\cdot}_h$ defines a norm on $V_h$.
\end{lemma}
\begin{proof}
\revision{It immediately follows from Lemma \ref{lem_norm_quiv}.}
\end{proof}

In the following main theorem, we derive an error equation to~\eqref{VEM3} to demonstrate how the VEM framework can, in a novel manner, overcome the difficulties of the non-conformity. \revision{This difficulty causes various issues for other DG-based or interior penalty-based approaches (c.f. Section \ref{sec:intro} and Remark \ref{rem_scaling}; see also the $\boldsymbol{H}^{1/2}$-penalty used in Ref.~\refcite{2005HoustonPerugiaSchneebeli}, and Theorem 2 in Ref.~\refcite{2016CasagrandeHiptmairOstrowski}).}
To reinstate the optimal rate of convergence, we need to further assume that the source term $\bff$ bears certain extra local regularity. 
\begin{theorem}
\label{thm_error_eqn}
Assume that $\bff\in L^2(\Omega)$ is locally in $\bfH^1$ around the interface, namely $\bff\in \bfH^1(K^q)$ on each $K^q \in \mathcal{T}^q_h$, and $\bfu\in \bfH^1(\curl;\Omega^-\cup\Omega^+)$ is the solution to \eqref{weak_form_1}. Let $\bs \bfv_h\in V_h$ be an arbitrary function in the VEM space, then for $\bfeta_h = \bfu_h - \bfv_h \in V_h$:
\begin{equation}
\begin{split}
\label{thm_error_eqn_eq0}
\vertiii{\bfeta_h}_h \lesssim &    \Big( \sum_{K^q\in\mathcal{T}^q_h} h_K^2 |f|^2_{H^1(K^q)}  + \sum_{K^q\in\mathcal{T}^q_h}h_K \| (\bfv_h - \Pi_h \bfv_h)\cdot\bft \|^2_{L^2(\partial K^q)} \Big)^{1/2}  \\
& + \| \alpha \, \emph{curl}\, \bfu - \alpha_h \curl \bfv_h \|_{L^2(\Omega^{\pm})} + \| \beta \bfu - \beta_h \Pi_h \bfv_h \|_{L^2(\Omega)} .
\end{split}
\end{equation}
\end{theorem}
\begin{proof}
We have
\begin{equation}
\label{thm_error_eqn_eq1}
a_h(\bfu_h, \bfeta_h)-  a_h( \bfv_h, \bfeta_h)  =\underbrace{ (\bff, \Pi_h \bfeta_h - \bfeta_h)_{\Omega} }_{(\rm I)} + \underbrace{ (\bff,  \bfeta_h)_{\Omega}-  a_h( \bfv_h, \bfeta_h) }_{(\rm{II})}.
\end{equation}
For $(\rm{ I})$, on all the triangular elements in $\mathcal{T}_h$, $\Pi_h$ reduce to identity operators, so $\Pi_h \bfeta_h - \bfeta_h$ simply vanishes. On a quadrilateral element $K^q\in\mathcal{T}^q_h$,
by the definition under~\eqref{VEM1},
$\Pi_h \bfeta_h=\Pi_{K^q}\bfeta_h$, which is the $L^2$-projection of $\bfeta_h $ on $K^q$.
Therefore
\begin{equation}
\begin{split}
\label{thm_error_eqn_eq2}
(\bff, \Pi_{K^q} \bfeta_h -  \bfeta_h)_{K^q} &
= (\bff - \Pi_{K^q} \bff, \Pi_{K^q} \bfeta_h -   \bfeta_h )_{K^q} \\
& \lesssim h_K |\bff|_{H^1(K^q)} \| \Pi_{K^q} \bfeta_h -  \bfeta_h \|_{L^2(K^q)}.
\end{split}
\end{equation}
For the term $(\rm{II})$ in~\eqref{thm_error_eqn_eq1}, using $ \underline{\curl} \, \alpha\curl \bfu + \beta\bs u = \bs f$, we have
\begin{equation}
\begin{split}
\label{thm_error_eqn_eq3}
 ({\rm II})
 = & \underbrace{ ( \underline{\curl} \, \alpha \curl \bfu,  \bfeta_h )_{\Omega} - (\alpha_h \, \curl  \bfv_h, \curl \bfeta_h )_{\Omega} }_{({\rm II}a)} \\
+ & \underbrace{ (\beta \, \bfu,    \bfeta_h)_{\Omega} - (\beta_h \, \Pi_h \bfv_h, \Pi_h \bfeta_h)_{\Omega} }_{({\rm II}b)} - \underbrace{ \sum_{K\in\mathcal{T}^{Bi}_h} S_K( \bfv_h, \bfeta_h)}_{({\rm II}c)}.
\end{split}
\end{equation}
For $({\rm II}a)$, since $\bfeta_h$ is in the conforming auxiliary space $V_h$ in~\eqref{auxi_v_space_3}, using the integration by parts, the continuity condition of the original PDE, and the curl condition in~\eqref{auxi_v_space_1}
we immediately have
\begin{equation}
\begin{split}
\label{thm_error_eqn_eq4}
({\rm II}a) & = (\alpha \, \curl \bfu, \curl  \bfeta_h )_{\Omega} - (\alpha_h \, \curl  \bfv_h, \curl  \bfeta_h )_{\Omega} \\
& \le \| \alpha \, \curl  \bfu - \alpha_h \, \curl \bfv_h \|_{L^2(\Omega)} \| \curl  \bfeta_h \|_{L^2(\Omega)}.
\end{split}
\end{equation}
For $({\rm II}b)$, on a triangular element $K^t$, we note that
\begin{equation}
\label{thm_error_eqn_eq5}
(\beta_h \, \Pi_h \bfv_h, \Pi_h \bfeta_h)_{K^t} = (\beta_h \, \bfv_h, \bfeta_h)_{K^t}.
\end{equation}
On a quadrilateral element $K^q$, by~\eqref{l2projection}
we have
\begin{equation}
\label{thm_error_eqn_eq6}
(\beta_h \, \Pi_{K^q} \bfv_h, \Pi_{K^q}\bfeta_h)_{K^q}
= (\beta_h \, \Pi_{K^q} \bfv_h,  \bfeta_h)_{K^q}.
\end{equation}
Combining~\eqref{thm_error_eqn_eq5} and~\eqref{thm_error_eqn_eq6}, we have
\begin{equation}
\label{thm_error_eqn_eq7}
({\rm II}b) = ( \beta \bfu - \beta_h \Pi_h \bfv_h, \bfeta_h )_{\Omega} \le \| \beta \bfu - \beta_h \Pi  \bfv_h \|_{L^2(\Omega)} \| \bfeta_h \|_{L^2(\Omega)} .
\end{equation}
In addition, for the stabilization term $({\rm II}c)$, by $( \bfeta_h - \Pi_{K^q} \bfeta_h )\cdot\bft_e\in \mathbb{P}_0(e)$ on each $e\subset\partial K^q$ and the definition of the interpolant in \eqref{interp_2}, we have
\begin{equation}
\begin{split}
\label{thm_error_eqn_eq6_1}
S_K( \bfv_h, \bfeta_h) & = h_K \int_{\partial K^q} (\bfv_h - \Pi_{K^q} \bfv_h )\cdot\bft ~ ( \bfeta_h - \Pi_{K^q}  \bfeta_h )\cdot\bft\, \dd s \\
& \le h_K \| (\bfv_h - \Pi_{K^q} \bfv_h )\cdot\bft \|_{L^2(\partial K^q)} \| ( \bfeta_h - \Pi_{K^q}  \bfeta_h )\cdot\bft \|_{L^2(\partial K^q)}.
\end{split}
\end{equation}
Finally, putting the estimates in~\eqref{thm_error_eqn_eq2}-\eqref{thm_error_eqn_eq6_1} to~\eqref{thm_error_eqn_eq1} yields the following bound
\begin{equation}
\begin{split}
\label{thm_error_eqn_eq8}
\vertiii{\bfeta_h}^2_h \lesssim & \Big( \sum_{K^q\in\mathcal{T}^q_h} h^2_K |f|^2_{H^1(K^q)} + \sum_{K^q\in\mathcal{T}^q_h}h_K \| (\bfv_h - \Pi_h \bfv_h)\cdot\bft \|^2_{L^2(\partial K^q)} 
\\
 + &  \| \alpha \, \curl  \bfu - \alpha_h \, \curl  \bfv_h \|^2_{L^2(\Omega)} 
+ \| \beta \bfu - \beta_h \Pi_h \bfv_h \|^2_{L^2(\Omega)}  \Big)^{1/2} 
\\
\cdot & \Big( \sum_{K^q\in\mathcal{T}^q_h} \big\{ \| \Pi_{K^q} \bfeta_h - \bfeta_h \|^2_{L^2(K^q)} 
+ \| ( \bfeta_h - \Pi_{K^q}  \bfeta_h )\cdot\bft \|^2_{L^2(\partial K^q)} \big\}
\\
& + \| \curl  \bfeta_h \|^2_{L^2(\Omega)} 
+ \| \bfeta_h \|^2_{L^2(\Omega)}   \Big)^{1/2}.
\end{split}
\end{equation}
The bound of $\| \Pi_{K^q} \bfeta_h - \bfeta_h \|_{L^2(K^q)}$ on quadrilateral elements follows from \eqref{lem_norm_quiv_eq1} in the coercivity of Lemma \ref{lem_norm_quiv}.
Putting the estimate above into~\eqref{thm_error_eqn_eq8}, and canceling one $\vertiii{\bfeta_h}$ on both sides yield the desired result.
\end{proof}

Then, in the later discussion, we shall let $\bfv_h$ be an interpolation $\bfu_I$, and thus the error can be decomposed into:
\begin{equation}
\label{err_decp}
\bfu-\bfu_h = \bfxi_h+\bfeta_h, ~~~
\text{where }~~ \bfxi_h = \bfu -\bfu_I,  ~~~ \text{and} ~~~ \bfeta_h = \bfu_I - \bfu_h,
\end{equation}

\section{Interpolation Error Estimates}
\label{sec:IFE_disre}
In this section, we estimate the interpolation errors and projection errors of virtual element spaces. Given any triangle $T$, the interpolation in \eqref{interp_2} exactly becomes the canonical edge interpolation \cite{2003Monk}. If $T$ is further assumed to be shape regular, then the following standard optimal approximation capability holds: 
\begin{equation}
\label{interp_1_err}
\| \bfu - I_T\bfu \|_{\bfH(\text{curl};T)} \lesssim h_T \| \bfu \|_{\bfH^1(\text{curl};T)}, ~~~~ \bfu \in \bfH^1(\text{curl};T).
\end{equation}

\subsection{Estimates based on the maximum angle condition}
\label{sec:est-max-angle}
Due to the assumption of the interface being smooth, we note that certain elements in $\mathcal{T}^t_h$ may inevitably have high aspect ratio in the process of mesh refining, which results that the commonly assumed shape regularity does not hold anymore. Consequently, \revision{the standard approximation results} of the edge interpolation \eqref{interp_1_err} cannot be directly applied. However, since maximum angles of triangles in the auxiliary triangulation around the interface
are uniformly bounded if the background mesh is shape regular, the interpolation error estimates can nevertheless be established based on the maximum angle condition. The interpolation estimates based on the maximum angle condition have been long studied for Lagrange elements \cite{1976BabuskaAziz,1992Michal}, Raviart-Thomas elements \cite{1999AcostaRicardo,2005BuffaCostabelDauge,1985Donatella}, and 3D N\'ed\'elec elements \cite{2005BuffaCostabelDauge}. 

\begin{lemma}[Lemma 2.3 and Theorem 4.1 in Ref.~\refcite{1999AcostaRicardo}]
\label{lem_maxangle}
Given any triangle $T$, let $\theta_T$ be the maximum angle of $T$, then
\begin{equation}
\label{lem_maxangle_eq0}
\| \bfu - I_T\bfu \|_{\bfH(\curl;T)} \lesssim \frac{h_K}{\sin(\theta_T)} \| \bfu \|_{\bfH^1(\curl;T)}, ~~~~ \bfu \in \bfH^1(\curl;T).
\end{equation}
\end{lemma}

The results above can be directly applied to estimate the interpolation errors of the virtual space $V_h(K^q)$ on $K^q\in\mathcal{T}^q_h$. Again we present in a more general setting.
\begin{lemma}\label{lem_auxi_interp}
Let $P$ be a simple polygon satisfying \hyperref[asp:polygon]{(P0)--(P2)} and let $I_P \bs u$ be the edge interpolation to $V_h(P)$ defined in \eqref{interp_2}. Then
\begin{equation}
\label{lem_auxi_interp_eq0}
\| \bfu - I_P \bs u \|_{\bfH(\curl;P)} \lesssim h_{P}\| \bfu \|_{\bfH^1(\curl;{\rm Conv}(P))}, ~~~~ \bfu \in \bfH^1(\curl;{\rm Conv}(P)),
\end{equation}
\revision{where ${\rm Conv}(P)$ is the convex hull of $P$.}
\end{lemma}
\begin{proof}
The estimate for the semi-curl norm is standard since $\text{curl} \, I_P \bs u$ is the $L^2$ projection of $\curl \bfu$ on $P$. Then 
\begin{align*}
\| \curl \bs u - \curl I_P \bs u \|_{L^2(P)} &=  \| \curl \bs u - \Pi_{P}\curl \bs u \|_{L^2(P)}\\
&\leq \| \curl \bs u - \Pi_{{\rm Conv}(P)}\curl \bs u \|_{L^2(P)}\\
&\leq \| \curl \bs u - \Pi_{{\rm Conv}(P)}\curl \bs u \|_{L^2({\rm Conv}(P))}\\
& \leq \frac{h_P}{\pi}\| \bfu \|_{\bfH^1(\curl;{\rm Conv}(P))},
\end{align*}
where the last step is the Poincar\'e inequality over convex domains \cite{Payne;Weinberger:1960optimal}.

Let $I_h$ be the edge interpolation to \revision{$\mathcal{ND}_h(\mathcal T_h(P))$}, i.e., the standard edge finite element space on mesh $\mathcal T_h(P)$. By the maximum angle condition in \hyperref[asp:polygon]{(P1)} and Lemma \ref{lem_maxangle}, we have $\| \bs u - I_h\bs u\|_{L^2(P)} \lesssim h_P\| \bfu \|_{\bfH^1(\curl;P)}$. Then it suffices to estimate the difference $\| I_P \bs u - I_h\bs u\|_{L^2(K)}$ on each triangle $K\in \mathcal T_h(P)$. We apply Lemma \ref{lem_Poincare} on each $K$ to get
$$
\| I_P \bs u - I_h\bs u\|_{L^2(K)} \leq \sum_{e\subset \partial K} h_K^{1/2} \| ( I_P \bs u - I_h\bs u)\cdot \bs t\|_{L^2(e)} + h_K\| \curl ( I_P \bs u - I_h\bs u)\|_{L^2(K)}.
$$
As $(I_P \bs u - I_h\bs u)\cdot \bs t = 0$ for $e\subset \partial P$, we only consider an interior edge $e$. Since $P$ is simple, any interior edge $e$ divides $P$ into two parts. Choose the part with less boundary edges and denoted by $P_e$, then we have the relation
$$
|e| (I_P \bs u - I_h\bs u)\cdot \bs t_e = \int_{\partial P_e} (I_P \bs u - I_h\bs u)\cdot \bs t \dd s= \int_{P_e} \curl ( I_P \bs u - I_h\bs u) \dd x,
$$
which can be used to get 
$$
|e|^{1/2} \| ( I_P \bs u - I_h\bs u)\cdot \bs t\|_{L^2(e)} \leq \| \curl ( I_P \bs u - I_h\bs u)\|_{L^2(P_e)} | P_e|^{1/2}.
$$
Using the triangle inequality and Assumption \hyperref[asp:NSIE]{(P3)}, together with the estimates for $\curl (\bs u-\bs I_P u)$ and $\curl (\bs u-\bs I_h u)$, we conclude for any $K\in \mathcal T_h(P)$
\begin{equation}\label{superclose}
\| I_P \bs u - I_h\bs u\|_{L^2(K)} \lesssim h_P^2\|\curl \bs u\|_{H^1({\rm Conv}(P))}.
\end{equation}
The desired result \eqref{lem_auxi_interp_eq0} then follows from the triangle inequality. 
\end{proof}

\subsection{An interface-aware interpolation}
\label{subsec:interp_def}
The standard local interpolation error estimate in Lemma \ref{lem_maxangle} is for the norm $\|\curl \bs u\|_{H^1(K)}$.  When $K$ is an interface element, in general $\curl \bs u\not\in H^1(K)$ but \revision{is in $H^1(K^+\cup K^{-})$}. 
Instead, we will use the fact $\curl \bs u_E^{\pm}\in H^1(K)$ and define the interpolation by the tangential components of either $\bs u_E^+$ or $\bs u_E^{-}$. The choice of whether to use $\bs u_E^+$ or $\bs u_E^{-}$ depends on the measure of $K^-$ and $K^+$. Note that, in the present situation, since both the triangular elements in $\mathcal{T}^t_h$ and the quadrilateral elements in $\mathcal{T}^q_h$ may have high aspect ratio, the modification in Ref.~\refcite{2012HiptmairLiZou} may not be suitable on anisotropic meshes with interface being present. Therefore, we shall employ a different interface-aware interpolation.

In the following discussion, we only present the results for the elements in the interface-approximated mesh $\mathcal{T}_h$ due to the technical treatment for the interface. 
Nonetheless, we emphasize that most of the results can be generalized based on the estimate of the interpolation errors on general polygons above. 
We shall use $K$ to denote an interface element in $\mathcal{T}^B_h$ that is cut into $K^-_h$ and $K^+_h$ by the edge $\Gamma_h^K$, and without loss of generality, we assume  $K^-_h\in \mathcal{T}^t_h$ and $K^+_h\in \mathcal{T}^q_h$. 
Recall that $K_{\rm int}$ is the portion sandwiched between $\Gamma$ and $\Gamma_h^K$, and we further define $K^{\pm}_{\rm int}:=K^{\pm}_h\cap K_{\rm int}$ which is equivalent to $K^{\pm}_h\cap K^{\mp}$, namely the mismatching subregions of $K^{\pm}_h$ as shown in Figure \ref{fig:interf_interp}. Let $\mathcal{E}_{K}$ be the collection of edges of $K^-_h$ and $K^+_h$ but excluding the edge $\Gamma^K_h$. We define a modified interpolation operator $\tilde{I}_{K}$ on $K\in\mathcal{T}^{Bi}_h$ such that
\begin{subequations}
\label{modify_interp}
\begin{align}
\int_e \tilde{I}_{K} \bfu\cdot\bft \dd s =\int_e \bfu\cdot\bft \dd s, ~~~~~ \forall e\in \mathcal{E}_{K}, \label{modify_interp_1} 
\end{align}
\begin{align}
\int_{\Gamma^K_h} \tilde{I}_{K} \bfu\cdot\bft \dd s =
\begin{cases}
      & \displaystyle\int_{\Gamma^K_h} \bfu^+_E\cdot\bft \dd s, ~~~~ \text{if} ~ |K^+_h| \leq |K^-_h|, 
      \\[3pt]
      & \displaystyle\int_{\Gamma^K_h} \bfu^-_E\cdot\bft \dd s, ~~~~ \text{if} ~ |K^-_h| < |K^+_h|. 
\end{cases}
\end{align}
\end{subequations}
By such a definition, we can always keep the interpolation as the standard one on the subelement with smaller size. So when estimation on the mismatch portion is needed, such as \eqref{lem_est_tri_eq2}, the element size appearing on the denominator will be always larger than $|K|/2$ such that the overall estimate can be controlled. This consideration serves as our key motivation to make this modification. For simplicity, we denote $\bfu_I$ by the global interpolant such that
\begin{equation}
\label{eq:def-uI-modified}
\bfu_I = I_K\bfu ~~~ \text{if} ~K\notin\mathcal{T}^{Bi}_h ~~~~ \text{and} ~~~~ \bfu_I = \tilde{I}_K\bfu ~~~ \text{if} ~K\in\mathcal{T}^{Bi}_h.
\end{equation}
In addition, we use $I_{K^{\pm}_h}\bfu^{\pm}_E$ to denote the canonical interpolation on $K^{\pm}_h$ for Sobolev extensions $\bfu^{\pm}_E$. We emphasize that the modified $\tilde I_K$ serves the purpose for the error analysis and is not needed in the actual computation.

The following two lemmas are presented for general polygons. So we temporarily let $K$ be an interface polygon, and the notation $\Gamma^K$, $\Gamma^K_h$ and $K_{\rm int}$ are all defined in the same manner as their counterparts for triangular interface elements.
For the subelement with larger size, inevitably there is a mismatch on $\Gamma^K_h$, so these results are essential. In the following discussion, with slightly abuse of the notations, we denote $\tilde{h}_K = |\Gamma^K_h|$ which might be much smaller than $h_K$ (see Fig. \ref{fig:interf_interp} (left)), and $\| \bfu_E^{\pm} \|_{L^2(K)}:= \| \bfu_E^- \|_{L^2(K)} + \| \bfu_E^{+} \|_{L^2(K)}$ with trivial generalization to other Sobolev norms. 

\begin{figure}[h]
\centering
   \includegraphics[width=2in]{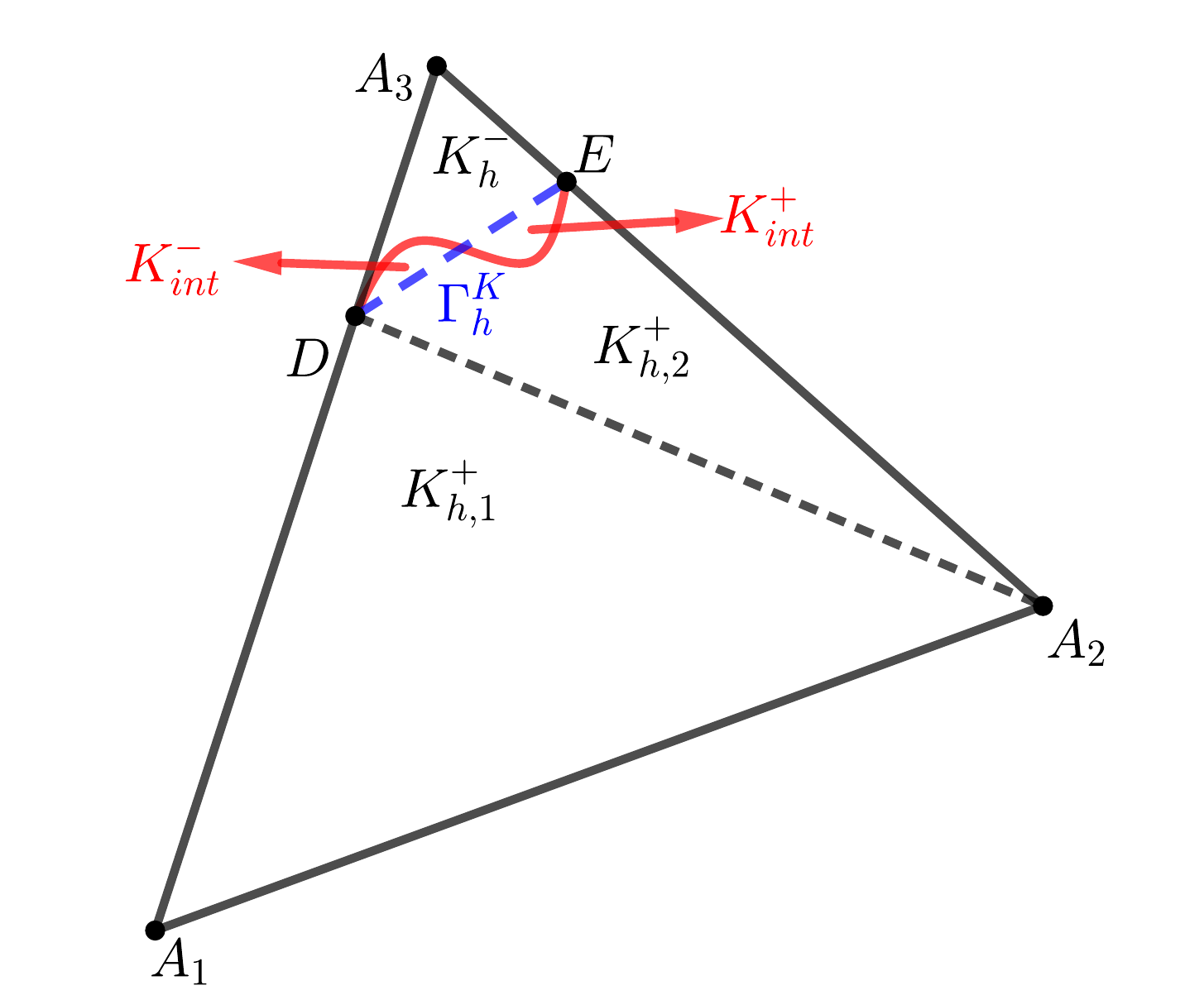}
   \includegraphics[width=2in]{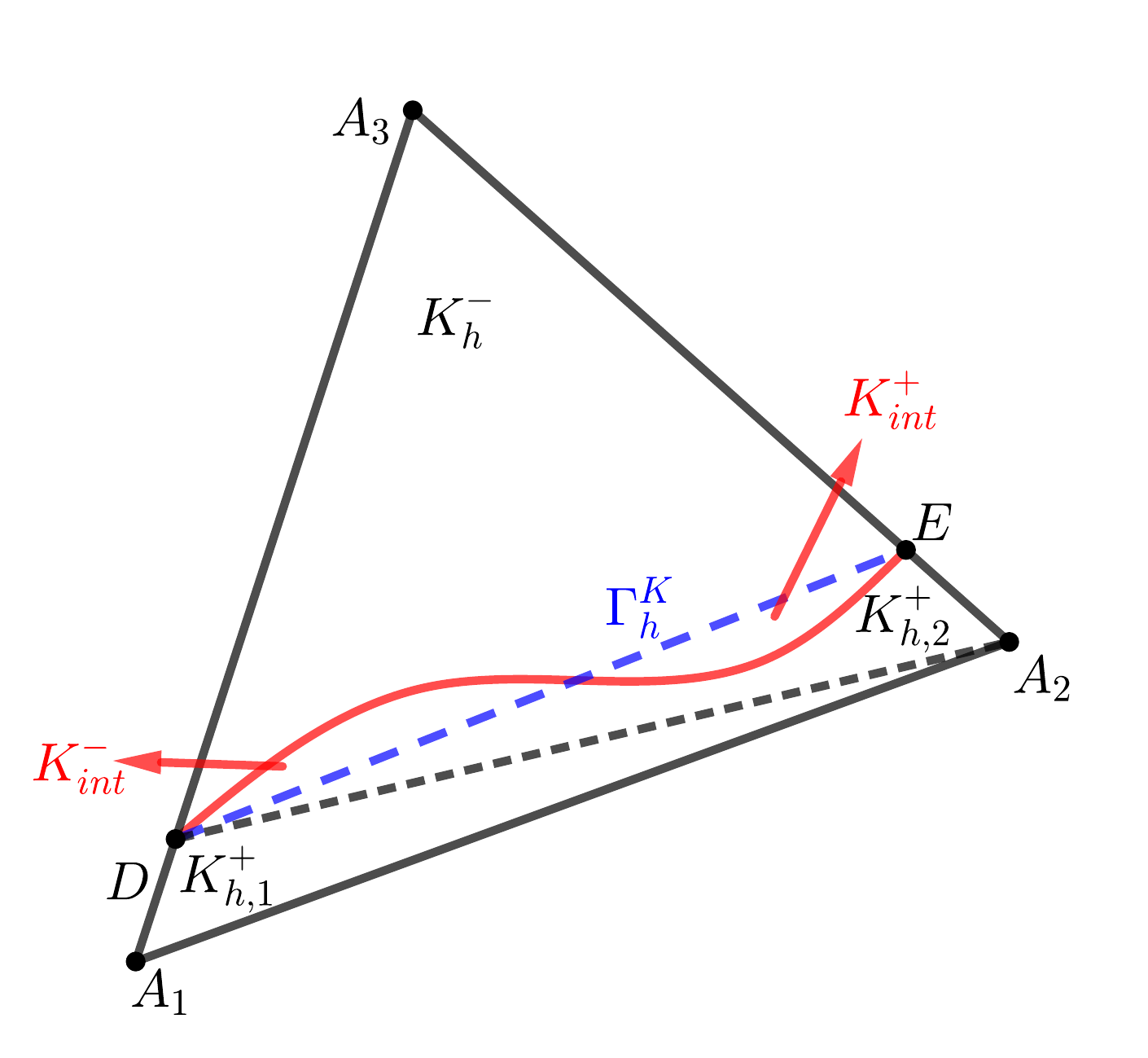}
  \caption{Left: the triangular interface is the smaller one. Right: the quadrilateral element is the smaller one.}
  \label{fig:interf_interp}
\end{figure}

The edge $\Gamma_h^K$ is assumed to be part in $\Omega^+$ and part in $\Omega^-$. As a result, the line integral $\int_{\Gamma_h^K}\bs u_I\cdot \bs t \dd s$ has part of the integrand being $\bs u_E^+\cdot \bs t$ while the other being $\bs u_E^-\cdot \bs t$. Their difference appears one of the key terms to bound the error of the modified interpolant \eqref{modify_interp}, as one adheres to one extension in defining the interpolation. 
\begin{lemma}
\label{lem_est_e}
Let $\bfu\in\bfH^1(\curl;\Omega^-\cup\Omega^+)$. Given an interface polygon $K$, there holds 
\begin{equation}
\label{lem_est_e_eq0}
\verti{ \int_{\Gamma^K_h} (\bfu^+_E - \bfu^-_E)\cdot\bft \dd s } \lesssim \tilde{h}_K^{1/2}h_K \| \curl \bfu^{\pm}_E \|_{L^2(K_{\rm int})}.
\end{equation}
\end{lemma}
\begin{proof}
Applying integration by parts on $K_{\rm int}$ and using the jump condition in \eqref{inter_jc_1}, we obtain
\begin{equation*}
\begin{split}
\verti{ \int_{\Gamma^K_h} (\bfu^+_E - \bfu^-_E)\cdot\bft \dd s } & = \verti{ \int_{K_{\rm int}} \text{curl}(\bfu^+_E - \bfu^-_E) \dd s } \lesssim |K_{\rm int}|^{1/2}\| \curl \bfu^{\pm}_E \|_{L^2(K_{\rm int})}
\end{split}
\end{equation*}
which yields \eqref{lem_est_e_eq0} since $|K_{\rm int}| \lesssim \tilde{h}_Kh^2_K$ by \eqref{lemma_interface_flat_eq1}.
\end{proof}

The result above can be used to derive the following trace inequality. Recall that there exists a shape regular triangle $B_h^K\subseteq \Omega$ with the base $\Gamma_h^K$ and a height $\mathcal{O}(h_K)$ by Assumption \hyperref[asp:background]{(B)} for all interface elements.
\begin{lemma}
\label{lem_est_e_L2}
Let $\bfu\in\bfH^1(\curl;\Omega^-\cup\Omega^+)$. Given an interface polygon $K$ with $\Gamma_h^K$, there holds
\begin{equation}
\begin{split}
\label{lem_est_e_L2_eq0}
\| (\bfu^+_E - \bfu^-_E)\cdot\bft \|_{L^2(\Gamma^K_h)} & \lesssim h^{1/2}_K\|  \bfu^{\pm}_E \|_{H^1( B_h^K)} + h_{K} \| \curl \bfu^{\pm}_E \|_{L^2(K_{\rm int})}.
\end{split}
\end{equation}
\end{lemma}
\begin{proof}
Apply the $L^2$-projection on $\Gamma^K_h$ to obtain
\begin{equation}
\begin{split}
\label{lem_est_e_L2_eq1}
 \| (\bfu^-_E -  \bfu^+_E )\cdot\bft \|_{L^2(\Gamma^K_h)} & \le \underbrace{ \| (\bfu^-_E -  \bfu^+_E )\cdot\bft - \Pi_{\Gamma^K_h}((\bfu^-_E -  \bfu^+_E )\cdot\bft) \|_{L^2(\Gamma^K_h)} }_{({\rm I})} \\
 & + \underbrace{ \| \Pi_{\Gamma^K_h}((\bfu^-_E -  \bfu^+_E )\cdot\bft) \|_{L^2(\Gamma^K_h)}}_{({\rm II})}.
\end{split}
\end{equation}
Since $\bft$ is a constant vector, and $\Gamma^K_h$ with $B_h^K$ satisfies the height condition, \revision{i.e., the height $l$ of $B_h^K$ supporting $\Gamma^K_h$ is $\mathcal{O}(h_K)$}, by the trace inequality (Lemma 6.3 in Ref.~\refcite{2018CaoChen}) and the Poincar\'e inequality with average zero on a boundary edge (Lemma 6.11 in Ref.~\refcite{2018CaoChen}), we have
\begin{equation}
\begin{split}
\label{lem_est_e_L2_eq2}
({\rm I}) & \lesssim l^{-1/2} \| (\bfu^-_E -  \bfu^+_E )\cdot\bft - \Pi_{\Gamma^K_h}((\bfu^-_E -  \bfu^+_E )\cdot\bft) \|_{L^2(B_h^K)} \\
& +  (l^{1/2} + l^{-1/2} \tilde{h}_K ) | (\bfu^-_E -  \bfu^+_E )\cdot\bft |_{H^1(B_h^K)} \\
& \lesssim   h_K^{1/2} | (\bfu^-_E -  \bfu^+_E )\cdot\bft |_{H^1(B_h^K)}.%
\end{split}
\end{equation}
For $({\rm II})$, by Lemma \ref{lem_est_e}, we have 
\begin{equation}
\begin{split}
\label{lem_est_e_L2_eq3}
({\rm II}) &= \tilde{h}_K^{1/2} \verti{ \Pi_{\Gamma^K_h}((\bfu^-_E -  \bfu^+_E )\cdot\bft) } \\
&=  \tilde{h}_K^{-1/2}  \verti{ \int_{\tilde{e}}(\bfu^-_E -  \bfu^+_E )\cdot\bft \dd s }  \lesssim h_{K} \| \curl \bfu^{\pm}_E \|_{L^2(K_{\rm int})}.
\end{split}
\end{equation}
Putting \eqref{lem_est_e_L2_eq2} and \eqref{lem_est_e_L2_eq3} back into \eqref{lem_est_e_L2_eq1} finishes the proof.
\end{proof}

\subsection{Estimate on interface elements}
Now we proceed to estimate the interpolation errors $\bfu - \bfu_I$ on interface elements for the modified interpolation.

\begin{lemma}
\label{lem_est_interface}
Let $\bfu\in\bfH^1(\curl;\Omega^-\cup\Omega^+)$ and $\bfu_{I}$ be the interpolant defined in \eqref{eq:def-uI-modified}. On each interface element $K\in \mathcal{T}^{Bi}_h$, there holds
\begin{align}
\label{lem_est_tri_eq0}
\| \bfu - \bfu_I \|_{\bfH(\curl;K)} \lesssim & \, h_{K}\| \bfu^\pm_E \|_{\bfH^1(\curl;K\cup B_h^K)} +  \| \bfu^{\pm}_E \|_{\bfH(\curl; K_{\rm int})}.
\end{align}
\end{lemma}
\begin{proof}
Recall that $K = K_h^{-}\cup K_h^+$. Without loss of generality, we focus the proof on $K_h^-$ as the estimate on the other part follows the result on  $K_h^-$ using a similar argument as the one in Lemma \ref{lem_auxi_interp}. By the triangle inequality, we have 
\begin{align}
\| \bfu - \bfu_I \|_{\bfH(\text{curl};K^-_h)}  \le & \,\| \bfu - \bfu^-_E \|_{\bfH(\text{curl};K^-_h)} \label{interface_interp_1}\tag{{\rm I}}  
\\
& + \| \bfu^-_E - I_{K^-_h} \bfu^-_E \|_{\bfH(\text{curl};K^-_h )} \label{interface_interp_2}\tag{{\rm II}} 
 \\
&+ \| I_{K^-_h} \bfu^-_E - \bfu_I \|_{\bfH(\text{curl};K^-_h )}. \label{interface_interp_3}\tag{{\rm III}} 
\end{align}
The first term \eqref{interface_interp_1} can be bounded by
\begin{equation}
\begin{split}
\label{lem_est_tri_eq1_1}
\| \bfu - \bfu^-_E \|_{\bfH(\text{curl};K^-_h)} & = \| \bfu_E^+ - \bfu^-_E \|_{\bfH(\text{curl};K^-_{\rm int})} \leq \| \bfu^{\pm}_E \|_{\bfH(\text{curl};K_{\rm int})}.
\end{split}
\end{equation}
The second term \eqref{interface_interp_2}'s estimate directly follows from Lemma \ref{lem_maxangle} since the triangular element $K^{-}_h$ satisfies the maximum angle condition. The third term \eqref{interface_interp_3} simply vanishes if $|K^-_h| \le | K^+_h |$, therefore the estimate for this term is only needed when $|K^-_h| > | K^+_h |$ and consequently $|K^-_h| \geq Ch_K^2$. For simplicity, we let $\bfw_h= I_{K^-_h} \bfu^-_E - \bfu_I$, and note that $\bfw_h\cdot\bft$ vanishes on the edges of $K^-_h$ except $\Gamma^K_h$. Using integration by parts and Lemma \ref{lem_est_e}, we have
\begin{equation}
\begin{aligned}
\label{lem_est_tri_eq2}
\| \text{curl} \, \bfw_h \|_{L^2(K^-_h)} &= |K^-_h|^{1/2} \verti{ \text{curl} \, \bfw_h }
= \frac{1}{|K^-_h|^{1/2} } \verti{ \int_{\Gamma^K_h} \bfw_h \cdot\bft \dd s } \\
& \lesssim \frac{1}{h_{K}} \verti{ \int_{\Gamma^K_h} ( \bfu^-_E -  \bfu^+_E) \cdot\bft \dd s } \lesssim h^{1/2}_{K} \| \curl \bfu^{\pm}_E \|_{L^2(K_{\rm int})}.
\end{aligned}
\end{equation}
To estimate the $L^2$-norm, we use inequality \eqref{lem_norm_equiv2_eq0} in Lemma \ref{lem_Poincare} to conclude
\begin{equation}
\label{lem_est_tri_eq3}
\| \bfw_h \|_{L^2(K^-_h)} \lesssim h^{1/2}_{K} \| \bs w_h \cdot \bs t\|_{L^2(\Gamma_h^K)} + h_K \| \curl \bs w_h \|_{L^2(K^-_h)}.
\end{equation}
Lastly, using Lemma \ref{lem_est_e_L2} and the bound of $\| \text{curl} \, \bfw_h \|_{L^2(K^-_h)} $ finishes the proof. 

\end{proof}

The estimates on non-interface elements in the background mesh $\mathcal{T}^B_h$ are standard.  These estimates together with the Sobolev inequality in Lemma \ref{lem_delta} and Theorem \ref{thm_ext} on the extension yield the global interpolation estimate.

\begin{theorem}
\label{thm_interp}
Let $\bfu\in\bfH^1(\curl;\Omega^-\cup\Omega^+)$, then there holds
\begin{equation}
\label{thm_interp_eq0}
\| \bfu -  \bfu_I  \|_{H(\curl; \Omega)} \lesssim h \|  \bfu \|_{\bfH^1(\curl;\Omega^-\cup\Omega^+)}. 
\end{equation}
\end{theorem}
\begin{proof}
For non-interface elements, the estimate is standard as well. For interface element $K$, we then use Lemma \ref{lem_est_interface}:
\begin{align*}
\sum_{K\in\mathcal{T}^{Bi}_h }\| \bfu - \bfu_I \|_{\bfH(\curl;K)}^2 \lesssim & \sum_{K\in\mathcal{T}^{Bi}_h } \, h_{K}^2\| \bfu^\pm_E \|_{\bfH^1(\curl;K\cup B_h^K)}^2 +  \| \bfu^{\pm}_E \|_{\bfH(\curl; K_{\rm int})}^2,\\
\lesssim & \, h^2 \| \bfu^\pm_E \|_{\bfH^1(\curl;\Omega^-\cup\Omega^+)}^2 +  \| \bfu^{\pm}_E \|_{\bfH(\curl; \cup_{K\in \mathcal{T}^{Bi}_h }K_{\rm int})}^2,
\end{align*}
in which the second step we use the fact $\Gamma_h$ is uniform Lipschitz so that the overlapping portions of triangles $B_h^K$ for every interface element $K$ are bounded. The desired estimate follows from Theorem \ref{thm_ext} and estimate \eqref{lem_delta_eq0}.
\end{proof}

\subsection{Estimate on the stabilization}
In this subsection, we move back to the mesh $\mathcal{T}_h$ consisting of triangular and quadrilateral elements cut from the background triangular mesh. 
On the quadrilateral $K_h^+$, a stabilization term is present. 
In this section, such terms that help to estimate the stabilization term in the error bound \eqref{thm_error_eqn_eq0} are estimated, including
\begin{equation*}
\label{stab_explain_1}
\| (\bfu - \bfu_I)\cdot\bft \|_{L^2(\partial K)}, ~~ \| \Pi_{K^q}(\bfu - \bfu_I)\cdot\bft \|_{L^2(\partial K)}, ~~ \| (\bfu - \Pi_{K^q} \bfu_I)\cdot\bft \|_{L^2(\partial K)}.
\end{equation*}
The main difficulty is on the second term above. Note that the common and natural approach to estimate the edge terms is to apply the trace inequality. Using Figure \ref{fig:interf_interp} as an example, it indeed works for the edges $A_1D$ and $A_2E$ as they support an $\mathcal{O}(h_K)$ height within the triangle. 
The major difficulty arises for edges like $A_1A_2$ and $\Gamma^K_h = DE$, due to a possibly degenerating height. The core idea of our approach is to employ a constructive proof, without relying on the trace inequality, to control the edge terms by using $\Pi_{K^q}(\bfu - \bfu_I)\cdot\bft$ being a constant for lifting and applying the definition of projection \eqref{compute_proj}. In the coming proofs, $\bfxi_h := \bfu -\bfu_I$ for simplicity. 
\begin{lemma}
\label{lem_est_qua_2}
Let $\bfu\in\bfH^1(\curl;\Omega^-\cup\Omega^+)$. Given each interface element $K\in \mathcal{T}^{Bi}_h$, there holds
\begin{equation}
\label{lem_est_qua_2_eq0}
\| (\bfu -  \bfu_I)\cdot\bft \|_{L^2(\partial K^+_h)} \lesssim h^{1/2}_K  \|  \bfu^{\pm}_E \|_{H^1( K)} + h_{K} \| \curl \bfu^{\pm}_E \|_{L^2(K_{\rm int})}.
\end{equation}
\end{lemma}
\begin{proof}
First, we have on each edge $e\neq\Gamma^K_h\subseteq\partial K^+_h$, $\int_e\bfxi_h \dd s =0 $, and thus
\begin{equation}
\label{lem_est_qua_2_eq1}
\|  \bfxi_h\cdot \bft \|_{L^2(e)} \lesssim h^{1/2}_e |\bfu\cdot\bft|_{H^{1/2}(e)} \le Ch^{1/2}_K \| \bfu^+_E \|_{H^1(K)},
\end{equation}
where we have used the fact that $e$ is one part of an edge of the regular element $K$, 
such that the trace inequality can be applied on this edge and $K$. 
On $\Gamma^K_h$, by the triangle inequality, we have
\begin{equation}
\begin{split}
\label{lem_est_qua_2_eq2}
\|  \bfxi_h\cdot \bft \|_{L^2(\Gamma^K_h)} \le &  \| (\bfu -  \bfu^+_E )\cdot\bft \|_{L^2(\Gamma^K_h)} \\
+ &  \| (\bfu^+_E - I_{K^+_h}\bfu^+_E )\cdot\bft \|_{L^2(\Gamma^K_h)} + \| ( I_{K^+_h}\bfu^+_E - \bfu_I )\cdot\bft \|_{L^2(\Gamma^K_h)}.
\end{split}
\end{equation}
For the first term in \eqref{lem_est_qua_2_eq2}, note that
\begin{equation}
\begin{split}
\label{lem_est_qua_2_eq3}
\| (\bfu -  \bfu^+_E )\cdot\bft \|_{L^2(\Gamma^K_h)}  
\le \| (\bfu^-_E -  \bfu^+_E )\cdot\bft \|_{L^2(\Gamma^K_h)} 
\end{split}
\end{equation}
of which the estimate follows from Lemma \ref{lem_est_e_L2}.
The second term in \eqref{lem_est_qua_2_eq2} 
follows from the argument similar to \eqref{lem_est_qua_2_eq1}. 
The third term in \eqref{lem_est_qua_2_eq2} simply vanishes when 
$|K^+_h|<|K^-_h|$. If $|K^+_h|\ge |K^-_h|$, then the estimate follows from Lemma \ref{lem_est_e}.
\end{proof}

\begin{lemma}
\label{lem_est_qua_3}
Let $\bfu\in\bfH^1(\curl;\Omega^-\cup\Omega^+)$. Given each interface element $K\in \mathcal{T}^{Bi}_h$, there holds
\begin{equation}
\label{lem_est_qua_3_eq0}
\| \Pi_{K^+_h} (\bfu - \bfu_I) \cdot \bft \|_{L^2(\partial K^+_h)} \lesssim h^{1/2}_K  \|  \bfu^{\pm}_E \|_{\bfH^1(\curl; K)} +  h^{-1/2}_{K} \| \bfu^{\pm}_E \|_{\bfH(\curl;K_{\rm int})}.
\end{equation}
\end{lemma}
\begin{proof}
For simplicity, in Figure \ref{fig:interf_interp}, we assume $A_1$ is at the origin, $K$ is contained in the first quadrant, and the edge $A_1A_2$ aligns with the $x_1$-axis having a tangential vector $(1,0)^\intercal$. Let $e$ be an edge of $K^+_h$ with the unit tangential vector $\bft_e$. 

If $e=A_1 D$ or $A_2 E$, since the height within $K^+_h$ with respect to these two edges cannot degenerate, a simple scaling directly leads to 
$$
\| \Pi_{K^+_h} \bfxi_h \cdot \bft \|_{L^2(e)} \lesssim h^{-1/2}_K  \| \Pi_{K^+_h} \bfxi_h \|_{L^2( K^+_h)} \leq h^{-1/2}_K  \| \bfxi_h \|_{L^2( K^+_h)}.
$$
Thus, the estimate follows from Lemma \ref{lem_est_interface}. 
If $e=A_1A_2$ or $DE$, when $|A_1D|\ge \gamma |A_1A_3|$ and $|A_2E| \ge \gamma |A_2A_3|$, with a constant $\gamma\in (0,1)$ bounded away from 0, the trace inequality-based argument above can be still applied. 

The major difficulty is how to deal with edges $e=A_1A_2$ or $DE$ when $K^+_h$ becomes degenerate. Without loss of generality, we assume $|A_1D|\le |A_1A_3|/2$ or $|A_2E| \le |A_2A_3|/2$ (see Figure \ref{fig:interf_interp} (right) for an illustration). In such a case, $|DE| \gtrsim h_K$ independent of the interface location thanks to the law of sines, as either $|A_3D|\ge |A_1A_3|/2$ or $|A_3E|\geq |A_3A_2|/2$. 
Now let $D=(r_1,r_2)$ and $E=(s_1,s_2)$, we have 
\begin{equation}
\label{lem_est_qua_3_eq1_1}
c\, h_K \max\{ r_2, s_2 \} \le |K^+_h| \le C\, h_K \max\{ r_2, s_2 \}. 
\end{equation}
Next, $p^e_h \in \mathbb{P}_1(K^+_h)$ is sought such that $\underline{\text{curl}}\,p^e_h = \bft_e$. Since $\Pi_{K^+_h} \bfxi_h$ is a constant, and by \eqref{compute_proj}, we have
\begin{equation}
\begin{split}
\label{lem_est_qua_3_eq1}
\| \Pi_{K^+_h} \bfxi_h\cdot \bft_e \|_{L^2(e)} & =\frac{ h_e^{1/2}}{ |K^+_h| } \verti{ ( \Pi_{K^+_h} \bfxi_h, \underline{\text{curl}}\, p^e_h )_{K^+_h} } \\
& \le  \underbrace{ \frac{ h_e^{1/2}}{ |K^+_h| }  \verti{ ( \curl   \bfxi_h, p^e_h )_{K^+_h} }}_{({\rm I})} +  \underbrace{ \frac{ h_e^{1/2}}{ |K^+_h| }  \verti{ (  \bfxi_h\cdot\bft, p^e_h )_{L^2(\partial K^+_h)} } }_{({\rm II})} .
\end{split}
\end{equation}
If $e=A_1A_2$, then $\bft_e = (1,0)^\intercal$, and $p^e_h = x_2$ which implies $\| p^e_h \|_{L^\infty(K^+_h)}\le \max\{ r_2, s_2 \} $.  If $e=DE$, then 
$$
\bft_e = \frac{( r_1 - s_1, r_2-s_2 )}{ |DE| },~~~ \text{and} ~~~ p^e_h = \frac{ ( r_1 - s_1)x_2 -  ( r_2 - s_2 )x_1 }{ |DE| }.
$$
Note that $|(r_1 - s_1)x_2| \lesssim h_K \max\{r_2,s_2\}$ and $|(r_2-s_2)x_1| \lesssim h_K \max\{r_2,s_2\}$ by the upper bound in \eqref{lem_est_qua_3_eq1_1}, as a result, the following estimate always holds
\begin{equation}
\label{lem_est_qua_3_eq1_2}
\| p^e_h \|_{L^{\infty}(K^+_h)} \lesssim \max\{r_2, s_2 \}.
\end{equation}
Now we proceed to estimate $({\rm I})$ and $({\rm II})$ individually. For $({\rm I})$, by the lower bound in \eqref{lem_est_qua_3_eq1_1} and \eqref{lem_est_qua_3_eq1_2} there holds
\begin{equation*}
\label{lem_est_qua_3_eq2}
\| p^e_h \|_{L^2(K^+_h)}\lesssim \| p^e_h \|_{L^{\infty}(K^+_h)} |K^+_h|^{1/2} \lesssim \max\{ r_2, s_2 \} |K^+_h|^{1/2} \lesssim |K^+_h|^{3/2}h^{-1}_K.
\end{equation*}
Therefore,
\begin{equation}
\label{lem_est_qua_3_eq3}
({\rm I}) \lesssim  |K^+_h|^{1/2}h^{-1/2}_K \| \curl  \bfxi_h \|_{L^2(K^+_h)} \lesssim h^{1/2}_K \| \curl  \bfxi_h \|_{L^2(K^+_h)} 
\end{equation}
of which the estimate follows from Lemma \ref{lem_est_interface}. 

For $({\rm II})$, on each $e'\subset \partial K^+_h$, using the same argument as above with \eqref{lem_est_qua_3_eq1_1} and \eqref{lem_est_qua_3_eq1_2}, we have $\| p^e_h \|_{L^2(e')} \lesssim |K^+_h| h^{-1/2}_K$.
Lastly, we arrive at
\begin{equation}
\label{lem_est_qua_3_eq5}
({\rm II}) \le \frac{ h_e^{1/2}}{ |K^+_h| }  \|  \bfxi_h\cdot\bft \|_{L^2(\partial K^+_h)} \| p^e_h \|_{L^2(\partial K^+_h)}  \lesssim \|  \bfxi_h\cdot\bft \|_{L^2(\partial K^+_h)}.
\end{equation}
The estimate of the right-hand side above follows from Lemma \ref{lem_est_qua_2}. Putting \eqref{lem_est_qua_3_eq3} and \eqref{lem_est_qua_3_eq5} into \eqref{lem_est_qua_3_eq1} finishes the proof.
\end{proof}

\begin{lemma}
\label{lem_conver_1}
Let $\bfu\in\bfH^1(\emph{curl};\Omega^-\cup\Omega^+)$. Given each interface element $K\in \mathcal{T}^{Bi}_h$, there holds
\begin{equation}
\begin{split}
\label{lem_conver_1_eq0}
 \| (\bfu - \Pi_{K^+_h} \bfu_I)\cdot\bft \|_{L^2(\partial K^+_h)} \lesssim & h^{1/2}_K \| \bfu^{\pm}_E \|_{\bfH^1(\emph{curl};K)} + h^{-1/2}_{K}  \| \emph{curl}\,\bfu^{\pm}_E \|_{L^2(K_{\rm int})}.
\end{split}
\end{equation}
\end{lemma}
\begin{proof}
First, the error is decomposed into 
\begin{equation}
\begin{split}
\label{lem_conver_1_eq1}
&  \| (\bfu - \Pi_{K^+_h} \bfu_I)\cdot\bft \|_{L^2(\partial K^+_h)}  \\
\le &  \| (\bfu - \Pi_{K^+_h} \bfu^+_E)\cdot\bft \|_{L^2(\partial K^+_h)}  +  \| (\Pi_{K^+_h} \bfu^+_E - \Pi_{K^+_h} \bfu_I )\cdot\bft \|_{L^2(\partial K^+_h)}.
\end{split}
\end{equation}
Here the estimate of the second term is similar to the one in Lemma \ref{lem_est_qua_3}. Therefore, we only need to estimate the first term in \eqref{lem_conver_1_eq1} which is further decomposed into
\begin{equation}
\begin{split}
\label{lem_conver_1_eq2}
& \| (\bfu - \Pi_{K^+_h} \bfu^+_E)\cdot\bft \|_{L^2(\partial K^+_h)} \\
 \le & \underbrace{ \| (\bfu -  \bfu^+_E)\cdot\bft \|_{L^2(\partial K^+_h)} }_{({\rm I})}+  \underbrace{ \| (\bfu^+_E - \Pi_{K^+_h} \bfu^+_E)\cdot\bft \|_{L^2(\partial K^+_h)} }_{({\rm II})}.
 \end{split}
\end{equation}
Note that $({\rm I})$ is only non-zero on $\Gamma_h^K$ of which the estimate follows form Lemma \ref{lem_est_e_L2}. For $({\rm II})$, if $e\subset \partial K^+_h$ is $A_1D$ or $A_2E$, i.e., it has an $\mathcal{O}(h_K)$ height within $K^+_h$. Then we apply the trace inequality (Lemma 6.3 in Ref.~\refcite{2018CaoChen}) and the approximation result of the $L^2$ projection to obtain
\begin{equation*}
\begin{split}
\label{lem_conver_1_eq3}
({\rm II}) \lesssim  h^{-1/2}_K \| \bfu^+_E - \Pi_{K^+_h}  \bfu^+_E  \|_{L^2( K^+_h)}  \lesssim h^{1/2}_K \| \bfu^+_E \|_{H^1(K^+_h)}.
 \end{split}
\end{equation*}
If $e\subset \partial K^+_h$ is $A_1A_2$ or $DE$ where the corresponding height may become degenerate,
we first apply the trace inequality on the whole shape-regular element $K$, and then apply the Poincar\'e inequality (e.g., see Lemma 5.3 in Ref.~\refcite{2018CaoChen}), to obtain
 \begin{equation*}
\label{lem_conver_1_eq4}
({\rm II}) \lesssim h^{-1/2}_K \| \bfu^+_E - \Pi_{K^+_h}  \bfu^+_E  \|_{L^2( K)} \lesssim h^{1/2}_K \| \bfu^+_E \|_{H^1(K)}.
\end{equation*}
Combining the estimates above finishes the proof.
\end{proof}

\section{Convergence Analysis}
\label{sec:convergence}
In this section, based on the previous results, we estimate the convergence order of the solution errors. In particular, we need to estimate each term in the error bound \eqref{thm_error_eqn_eq0}. Our main task is to estimate those terms on quadrilateral elements. In the following discussion, we still keep our notation that $K^+_h\in\mathcal{T}^q_h$ will be the quadrilateral subelement associated with each interface element $K\in\mathcal{T}^{Bi}_h$.

\begin{theorem}[An \textit{a priori} convergence result for VEM]
\label{thm_error}
Under the same assumption of Theorem \ref{thm_error_eqn}, let $\bfu\in\bfH^1(\emph{curl};\Omega^-\cup\Omega^+)$ and let the background mesh $\mathcal{T}^B_h$ satisfy the assumptions $(A)$ and $(B)$, then the solution $\bfu_h$ to the VEM scheme \eqref{VEM3} admits the error estimates
\begin{equation}
\label{thm_error_eq0}
\| \bs u - \bs u_h \|_{H(\curl; \Omega)} \lesssim h \| \bfu \|_{\bfH^1(\emph{curl};\Omega^+\cup \Omega^-)} + h \Big (\sum_{K^q\in\mathcal{T}^q_h}  |f|_{H^1(K^q)}^2 \Big )^{1/2}.
\end{equation}
\end{theorem}
\begin{proof}
First, by the triangle inequality, we have 
$$
\| \bs u - \bs u_h \|_{H(\curl; \Omega)} \leq \| \bs u - \bs u_I \|_{H(\curl; \Omega)} + \| \bs u_I - \bs u_h \|_{H(\curl; \Omega)}. 
$$
Recall $\bs \eta_h = \bs u_I - \bs u_h$. We use Lemma \ref{lem_Poincare} to obtain
\begin{align*}
\sum_{K\in \mathcal T_h}\|\bs \eta_h \|_K^2 & \le \sum_{K\in \mathcal T_h} \| \Pi_h \bs \eta_h\|_K^2 + \| (I - \Pi_h) \bs \eta_h\|_K^2\\
&\lesssim  \sum_{K\in \mathcal T_h} \| \Pi_h \bs \eta_h\|_K^2 + h_K\| (I - \Pi_h) \bs \eta_h\cdot \bs t \|_{\partial K}^2 + h_K^2\|\curl \bs \eta_h\|^2_K\\
&\lesssim \vertiii{\bfeta_h}_h^2
\end{align*}
Recall that, in Theorem \ref{thm_error_eqn}, we have obtained
\begin{equation}
\begin{split}
\vertiii{\bfeta_h}_h \lesssim &  \left (\sum_{K^q\in\mathcal{T}^q_h} h_K^2 |f|_{H^1(K^q)}^2 + h_K \| (\bfu_I - \Pi_h \bfu_I)\cdot\bft \|_{L^2(\partial K^q)}^2 \right )^{1/2}   \\
& + \| \alpha \, \curl  \bfu - \alpha_h \curl \bfu_I \|_{L^2(\Omega^{\pm})} + \| \beta \bfu - \beta_h \Pi_h \bfu_I \|_{L^2(\Omega)}  .
\end{split}
\end{equation}
Applying Lemma \ref{lem_conver_1} on the stabilization term, Theorem \ref{thm_interp} on the last two terms, together with Lemma \ref{lem_delta} and estimate \eqref{delta_strip_2} by acknowledging that $\alpha_h$, $\beta_h$ are different from $\alpha$, $\beta$ only on $K_{\rm int}$, we have proved the desired estimate.

\end{proof}

\section{Numerical Examples}

In this section, we present a group of numerical experiments to validate the previous estimates. Let the computation domain be $\Omega=(-1,1)\times(-1,1)$, and \revision{the background mesh} be generated by triangulating an $N\times N$ Cartesian mesh by cutting each square into two triangles along its diagonal. We highlight that the proposed method can be used on any other regular background triangular meshes. A circular interface $\{\Gamma:x^2+y^2=r^2_1\}$ cuts $\Omega$ into the inside subdomain $\Omega^-$ and the outside subdomain $\Omega^+$. We consider the following example \cite{2020GuoLinZou,2010LiMelenkWohlmuthZou} \revision{with the exact solution}
\begin{equation}
\label{exact_solu}
\bfu = 
\begin{cases}
      & \left(\begin{array}{c} \mu^-\left(  - k_1(r_1^2 - x^2 -y^2)y \right) \\ \mu^-\left(  -k_1(r_1^2 - x^2 -y^2)x  \right) \end{array}\right) ~~~~ \text{in}~ \Omega^-,  \\
      & \left(\begin{array}{c} \mu^+\left(  - k_2(r_2^2-x^2-y^2)(r_1^2-x^2-y^2)y  \right) \\  \mu^+\left( - k_2(r_2^2-x^2-y^2)(r_1^2-x^2-y^2)x \right) \end{array}\right) ~~~~ \text{in}~ \Omega^+,
\end{cases}
\end{equation}
where the boundary conditions and the right-hand side $\bff$ are calculated accordingly. We employ the parameters $k_2=20$, $k_1=k_2(r_2^2-r_1^2)$ with $r_1=\pi/5$ and $r_2=1$, and fix $\alpha^-=\beta^-=1$ with varying $\alpha^+=10$ or $100$ and $\beta^+=10$ or $100$. For simplicity, we define the errors
\begin{equation}
\label{error_num}
e_0 := \| \bfu - \Pi_h \bfu_h \|_{L^2(\Omega)} ~~~ 
\text{and} ~~~ e_1 := \| \curl(\bfu - \bfu_h )\|_{L^2(\Omega)}.
\end{equation}
The numerical results are presented in Tables \ref{table:alpha10beta10}-\ref{table:alpha100beta100}, and they clearly show the optimal first order convergence for both errors with respect to $h=1/N$.

\begin{table}[htp]
\centering
\setlength\tabcolsep{5pt}
\begin{minipage}{0.48\textwidth}
\centering
\begin{tabular}{|c |c | c|c | c|}
\hline
$h$   & $ e_0 $ & rate   & $e_1$ & rate   \\ \hline
1/10    & 0.6257                 &  NA      & 1.3893              & NA       \\ \hline
1/20    & 0.3258                 &  0.94    & 0.6998               & 0.99    \\ \hline
1/40    & 0.1661                 &  0.97    & 0.3534               & 0.99   \\ \hline
1/80    & 0.0843                 &  0.98    & 0.1784               & 0.99   \\ \hline
1/160  & 0.0424                 &  0.99    & 0.0894               & 1.00  \\ \hline
1/320   &0.0213                 &  1.00    & 0.0447               & 1.00   \\ \hline
1/640   &0.0107                 &  0.99    & 0.0224               & 1.00   \\ \hline
\end{tabular}
\caption{Solution errors for $\alpha^+=10$ and $\beta^+=10$.}
\label{table:alpha10beta10}
\end{minipage}
\hfill
\begin{minipage}{0.48\textwidth}
\centering
\begin{tabular}{|c |c | c|c | c|}
\hline
$h$   & $ e_0 $ & rate   & $e_1$ & rate   \\ \hline
1/10    &  0.6206                 &  NA      & 1.3912             & NA       \\ \hline
1/20    &  0.3257                 &  0.93    & 0.7000              & 0.99    \\ \hline
1/40    &  0.1661                 &  0.97    & 0.3534              & 0.99   \\ \hline
1/80    &  0.0843                 &  0.98    & 0.1784              & 0.99   \\ \hline
1/160  &  0.0424                 &  0.99    & 0.0894              & 1.00  \\ \hline
1/320  &  0.0213                 &  1.00    & 0.0447              & 1.00   \\ \hline
1/640  &  0.0107                 &  1.00    & 0.0224              & 1.00   \\ \hline
\end{tabular}
\caption{Solution errors for $\alpha^+=10$ and $\beta^+=100$.}
\label{table:alpha10beta100}
\end{minipage}
\end{table}

\begin{table}[htp]
\centering
\setlength\tabcolsep{5pt}
\begin{minipage}{0.48\textwidth}
\centering
\begin{tabular}{|c |c | c|c | c|}
\hline
$h$   & $ e_0 $ & rate   & $e_1$ & rate   \\ \hline
1/10    & 0.3266                 &  NA      & 1.0795              & NA       \\ \hline
1/20    & 0.1761                 &  0.89    & 0.5449               & 0.99   \\ \hline
1/40    & 0.0926                 &  0.93    & 0.2768               & 0.98  \\ \hline
1/80    & 0.0482                 &  0.94    & 0.1406               & 0.98  \\ \hline
1/160  & 0.0246                 &  0.97    & 0.0705               & 1.00 \\ \hline
1/320   &0.0124                 &  0.99    & 0.0353               & 1.00  \\ \hline
1/640   &0.0062                 &  0.99    & 0.0177               & 1.00  \\ \hline
\end{tabular}
\caption{Solution errors for $\alpha^+=100$ and $\beta^+=10$.}
\label{table:alpha100beta10}
\end{minipage}
\hfill
\begin{minipage}{0.48\textwidth}
\centering
\begin{tabular}{|c |c | c|c | c|}
\hline
$h$   & $ e_0 $ & rate   & $e_1$ & rate   \\ \hline
1/10    &  0.1938                &  NA      & 0.8877             & NA       \\ \hline
1/20    &  0.1425                &  0.44    & 0.4358              & 1.03    \\ \hline
1/40    &  0.0698                &  1.03    & 0.1976              & 1.14   \\ \hline
1/80    &  0.0368                &  0.92    & 0.1027              & 0.95   \\ \hline
1/160  &  0.0189                &  0.96    & 0.0503              & 1.03  \\ \hline
1/320  &  0.0101                &  0.90    & 0.0264              & 0.93   \\ \hline
1/640  &  0.0054                &  0.91    & 0.0140              & 0.92   \\ \hline
\end{tabular}
\caption{Solution errors for $\alpha^+=100$ and $\beta^+=100$.}
\label{table:alpha100beta100}
\end{minipage}
\end{table}

\section*{Acknowledgment}
This work was supported in part by the National Science Foundation under grants DMS-1913080, DMS-2012465, and DMS-2136075.

\end{document}